\documentclass{article}
\usepackage[utf8]{inputenc}
\usepackage{amssymb,amsmath,amsthm,amsfonts,latexsym}

\usepackage{tikz}
\usepackage{graphicx, color, subfigure}
\usepackage{siunitx}
\usepgflibrary{arrows}
\newtheorem{theorem}{Theorem}[section]
\newtheorem{lemma}[theorem]{Lemma}

\newtheorem{definition}[theorem]{Definition}

\usepackage{caption}
\usepackage{afterpage}
\usepackage{fullpage}
\usepackage{mathtools}
\usepackage[mathscr]{eucal}
\usepackage{pdfpages}
\usepackage[T1]{fontenc}
\usepackage[english]{babel}
\usepackage{xspace}
\usepackage{amscd,bm}
\usepackage{xcolor,soul}
\usepackage{verbatim}
\numberwithin{equation}{section}

\theoremstyle{remark}
\newtheorem{ex}{Example}

\title{Bound-preserving discontinuous Galerkin methods with modified Patankar time integrations for chemical reacting flows\footnote{\baselineskip 1.2pc The first and third authors were supported by the NSF grant DMS-1818467 and Simons Foundation 961585.}}
\author{Fangyao Zhu\footnote{Department of Mathematical Sciences,
Michigan Technological University, Houghton, MI 49931. E-mail: fangyaoz@mtu.edu},\quad Juntao Huang\footnote{Department of Mathematics and Statistics, Texas Tech University, Lubbock, TX 79409.
{Email: juntao.huang@ttu.edu}},\quad Yang Yang\footnote{Department of Mathematical Sciences,
Michigan Technological University, Houghton, MI 49931. E-mail:
yyang7@mtu.edu}}
\date{}
\begin{document}
\baselineskip 1.2pc
\maketitle
\begin{abstract}
\baselineskip 1pc
In this paper, we develop bound-preserving discontinuous Galerkin (DG) methods for chemical reactive flows. There are several difficulties in constructing suitable numerical schemes. First of all, the density and internal energy are positive, and the mass fraction of each species is between 0 and 1. Secondly, due to the rapid reaction rate, the system may contain stiff sources, and the strong-stability-preserving explicit Runge-Kutta method may result in limited time step sizes. To obtain physically relevant numerical approximations, we apply the bound-preserving technique to the DG methods. Though traditional positivity-preserving techniques can successfully yield positive density, internal energy and mass fractions, it may not enforce the upper bound 1 of the mass fractions. To solve this problem, we need to (1) make sure the numerical fluxes in the equations of the mass fractions are consistent with that in the equation of the density; (2) choose conservative time integrations such that the summation of the mass fractions is preserved. With the above two conditions, the positive mass fractions have summation 1, then they are all between 0 and 1. For time discretization, we apply the modified Runge-Kutta/multi-step Patankar methods, which are explicit for the flux while implicit for the source. Such methods can handle stiff sources with relatively large time steps, preserve the positivity of the target variables, and keep the summation of the mass fractions to be 1. Finally, it is not straightforward to combine the bound-preserving DG methods and the Patankar time integrations. The positivity-preserving technique for DG method requires positive numerical approximations at the cell interfaces, while Patankar methods can keep the positivity of the pre-selected point-values of the target variables. To match the degree of freedom, we use $Q^k$ polynomials on rectangular meshes for problems in two space dimensions. To evolve in time, we first read the polynomials at the Gaussian points. Then suitable slope limiters can be applied to enforce the positivity of the solutions at those points, which can be preserved by the Patankar methods, leading to positive updated numerical cell averages. In addition, we use another slope limiter to get positive solutions used for the bound-preserving technique for the flux. Numerical examples are given to demonstrate the good performance of the proposed schemes.\\

\end{abstract}

\vfill

\section{Introduction}

In this paper, we develop numerical methods for the chemical reacting flows \cite{1D example}
\begin{subequations}\label{eq11}
\begin{eqnarray}
&& \rho_t + m_x + n_y = 0   \label{1.1a} \\
&& m_t + (mu +p)_x + (mv)_y= 0  \label{1.1b} \\
&& n_t + (nu)_x + (nv + p)_y= 0  \label{1.1n} \\
&& E_t + ((E+p)u)_x + ((E+p)v)_y= 0  \label{1.1c} \\
&& (c_1)_t + (mz_1)_x + (nz_1)_y= s_1 \label{1.1d} \\
&& ... \notag \\ 
&& (c_{M-1})_t + (mz_{M-1})_x + (nz_{M-1})_y =s_{M-1}.\label{1.1e}
\end{eqnarray}
\end{subequations}
Here $\rho$ is the density, $(u,v)$ is the velocity field, $m=\rho u$ and $n=\rho v$ are the momentum in $x$ and $y$ directions, $E$ is the total energy which is the summation of the kinetic and internal energy of the fluid, $p$ is the pressure, $z_i$ is the mass fraction of the $i$-th species with $i=1,2, ...,M$ and $M$ being the number of species. For $1\leq i\leq M$, we have $c_i=\rho z_i$ and $\sum_{i=1}^{M} z_i=1$. Therefore, the total density satisfies $\rho=\sum_{i=1}^{M} c_i$ and $0\leq z_i\leq 1$. The pressure is
$$
p=RT\sum_{i=1}^M \frac{\rho_i}{M_i}
$$
where $R$ is the universal gas constant, $T=\rho/p$ is the temperature, and $M_i$ is the molar mass of the $i$-th species. The total energy $E$ is expressed as
$$
E = \sum_{i=1}^{M} \rho_i e_{in,i}(T) +\sum_{i=1}^{M} \rho_i q_i +\frac12 \rho (u^2 + v^2)
$$
where $q_i$ is the enthalpy for the $i$th species and $e_{in,i}(T)=C_i T$ is the internal energy of the $i$th species with $C_i=3R/2M_i$ and $5R/2M_i$ for the monoatomic and diatomic species, respectively.

We can write \eqref{eq11} in a compact form :
\begin{eqnarray}
&& \bm{U} _t+\bm{F}\left(\bm{U}\right)_x+\bm{G}\left(\bm{U}\right)_y= \bm{S}\left(\bm{U}\right) \label{equation2}
\end{eqnarray}
where $\bm{S}\left(\bm{U}\right)$ denotes the source describing the chemical reactions of the form
$$
\nu_{1,j}^{'}X_1+\nu_{2,j}^{'}X_2+...+\nu_{M,j}^{'}X_M ~ {\rightleftharpoons} ~ \nu_{1,j}^{''}X_1+\nu_{2,j}^{''}X_2+...+\nu_{M,j}^{''}X_M ,\quad j=1,...,J,
$$
where $J$ is the total number of reactions, $\nu_{i,j}^{'}$ and $\nu_{i,j}^{''}$ are the stoichiometric coefficients of the reactants and productions of the $i$th species in the $j$th reaction. For non-equilibrium chemistry, the rate of production in \eqref{eq11} can be written as 
$$
s_i=M_i\sum_{j=1}^{J}(\nu_{i,j}^{''}-\nu_{i,j}^{'})(k_{f,j}\prod_{s=1}^{M}\left(\frac{c_s}{M_s}\right)^{\nu_{s,j}^{'}}-k_{b,j}\prod_{s=1}^{M}\left(\frac{c_s}{M_s}\right)^{\nu_{s,j}^{''}}),\quad i=1,...,M.
$$
where $k_{f,j}$ and $k_{b,j}$ are functions of temperature indicating the forward and backward reaction rates. 

The solution to chemical reacting flows may contain shocks. Moreover, numerical schemes may result in non-physical numerical approximations, e.g. the density and pressure are negative, and the mass fractions are out of the interval $[0,1]$. The non-physical numerical approximations may further lead to ill-posed problems and eventually blow-up of the numerical simulations. Hence, constructing a bound-preserving scheme is essential. In this paper, we apply discontinuous Galerkin (DG) method that is high order accurate and flexible on geometry. The method was first introduced in 1973 \cite{DG} for the neutron transport equation, a time independent hyperbolic equation. Later, Cockburn and Shu extended the DG method for solving time dependent problems such as nonlinear convection problems and Euler equations. The framework was given in a series of papers \cite{Cockburn I,Cockburn II, Cockburn III, Cockburn IV, Cockburn V}, where the DG method was coupled with Runge-Kutta (RK) time integration along with TVB nonlinear limiters to achieve non-oscillatory properties for strong shocks. However, the TVB limiter is not sufficient to maintain the positivity of the numerical approximations. %
In \cite{PP/detonation}, high order DG methods for two-dimensional gaseous detonations were constructed to preserve the positivity of density, pressure, and all the mass fractions. The idea is to apply first order Euler forward time discretization and find a sufficient condition for the cell averages of the DG numerical approximations to be positive. Then a slope limiter is applied to construct new physically relevant numerical approximations, keeping the original cell averages. %
The time discretizations can be extended to high order strong-stability-presering (SSP) RK or multi-step (MS) methods \cite{SSP,Shu-Osher form,TVD} since they are convex combinations of the Euler forward method. %
For our problem, we need to preserve not only the lower bound 0 but also the upper bound 1 for the mass fractions. %
The first work preserving the physical bounds of the mass fractions was given in \cite{bound-preserving,Chuenjarern}, where the compressible miscible displacements in porous media were discussed. Later, the idea was adapted to multi-species and multi-reaction detonations to construct high order DG schemes in \cite{JDdetonation1,JDdetonation2}. The basic strategy is to apply the positivity-preserving techniques to each mass fraction $z_i$ and enforce $\sum_{i=1}^M z_i=1$ to obtain physically relevant numerical approximations by using conservative time integrations. The extension to finite difference methods was also given in \cite{DuCAMC}.

Another difficulty is how to deal with the stiff source terms due to the rapid reaction rate in chemical reactive flows. Direct application of the explicit time integration may result in limited time steps, and one may consider Implicit-Explicit (IMEX) methods which treats the source term implicitly while the flux term explicitly. By doing so, the positivity-preserving technique can also be applied to the flux terms. Unfortunately, it is not easy to construct implicit solvers for the source terms if the reactions are complicated and the system turns out to be fully nonlinear. Therefore, one may try to linearize the source term, keeping the conservation of mass. One method that can preserve the physical bounds and treat the source term implicitly is the modified Patankar (MP) time integration, which was developed from the Patankar trick that first introduced in \cite{Patankar trick}.

To demonstrate the idea of the MP time integration, we consider a class of system of ODEs which describes the chemical reactions by ignoring the convection terms in \eqref{equation2}:
\begin{eqnarray}
&& \frac{dc_i}{dt}=M_i\sum_{r=1}^{R}(\nu_{i,r}^{''}-\nu_{i,r}^{'})(k_{f,r}\prod_{s=1}^{M}\left(\frac{c_s}{M_s}\right)^{\nu_{s,r}^{'}}-k_{b,r}\prod_{s=1}^{M}\left(\frac{c_s}{M_s}\right)^{\nu_{s,r}^{''}}),\quad i=1,...,M.\label{equation3}
\end{eqnarray}
Equation \eqref{equation3} can be seen as production-destruction equations which have the form
\begin{eqnarray}
&& \frac{dc_i}{dt}=P_i(c)-D_i(c), \quad i=1,2,...,M, \label{equation4}
\end{eqnarray}
with
$$
P_i(c)=\sum_{j=1}^{M}p_{ij}(c), \quad D_i(c)=\sum_{j=1}^{M}d_{ij}(c)
$$
and
$$
p_{ij}(c)=d_{ji}(c)\geq0
$$
where $c=(c_1,c_2,...,c_M)^T$ denotes the concentrations, the production function $p_{ij}(c)$ describes the transformation rate of $j$-th component to $i$-th component; whereas $d_{ij}(c)$, the destruction function, denotes the transformation rate of $i$th component to $j$-th component. The solution of \eqref{equation4} is conservative which means that $\sum_{i=1}^{N}c_i(t)$ remains unchanged with respect to time. In addition, the solution should be positive if the initial condition is positive and $d_{ij}(c)=0$ for $c_i=0$.

The Patankar trick \cite{Patankar trick} was to multiply the destruction term by the ratio of the numerical approximations between two time levels, leading to a linear implicit scheme. Unfortunately, such a method fails the mass conservation. In contrast, the MP scheme in \cite{modified Patankar} guarantees both positivity and conservation properties with any time step size, and the method reads as:
\begin{eqnarray}
c_i^{n+1}=c_i^{n}+\Delta t\left(\sum_{j}p_{ij}(c^n)\frac{c_j^{n+1}}{c_j^n}-\sum_{j}d_{ij}(c^n)\frac{c_i^{n+1}}{c_i^n}\right).\label{equation6}
\end{eqnarray}
However, the scheme is first order accurate only. In \cite{MPRK2_PP,MPRK3_PP}, MP RK schemes of second and third order were introduced. In \cite{DeC}, the MP trick is adapted to deferred correction (DeC) schemes and developed MPDeC schemes of arbitrary order of accuracy. More recently, the stability analysis of the MP schemes are given in \cite{Izgin1,Izgin2,Izgin3,Izgin4}.
However, it is not easy to couple the convection terms into the schemes in \cite{MPRK2_PP}, \cite{MPRK3_PP} and  \cite{DeC}. Therefore, instead of using the classical form of RK methods, Huang and Shu applied RK schemes of Shu-Osher form \cite{Shu-Osher form} and constructed another class of MP RK schemes with the SSP structure. These conservative and unconditionally positivity-preserving MP RK methods were given in \cite{MPRK2,MPRK3} along with the necessary and sufficient conditions derived to obtain the desired order of accuracy. In these two works, the time integration was combined with the positivity-preserving finite difference weighted essentially non-oscillatory (WENO) scheme.

In this paper, we first construct the MS MP time integration and couple the RK/MS MP time integration with DG spatial discretization. The combination is not straightforward due to the inconsistency of the collocation points required in the positivity-preserving technique and the MP time integrations. Due to the SSP structure of the MP time integration to be used in this paper, we use \eqref{equation6} to demonstrate the full algorithm, and the fully-discrete scheme with convection term can be written as
$$
c_i^{n+1}=c_i^{n}+\Delta tf(c^n)+\Delta t\left(\sum_{j}p_{ij}(c^n)\frac{c_j^{n+1}}{c_j^n}-\sum_{j}d_{ij}(c^n)\frac{c_i^{n+1}}{c_i^n}\right),
$$
where $f(c)$ is the convection term after DG spatial discretization. Given the physically relevant numerical approximations at time level $n$, we first treat the convection term only, and define $\tilde{c}_i=c_i^n+\Delta tf(c^n)$. The positivity-preserving technique yields positive numerical cell average of $\tilde{c}$. Then we apply a slope limiter to obtain positive $\tilde{c}$ at Gaussian points in each cell as this is required by the MP time integration. Next, we update the numerical approximation at each Gaussian point by using the MP time integration, and the resulted numerical approximations are also positive, leading to positive numerical cell averages. Finally, another slope limiter is called to have positive numerical approximations at the cell interfaces which will be used in the next time level.

The paper is organized as follows. In Section 2, we first demonstrate the MP RK method and then develop the second and third order MS MP scheme and show the necessary and sufficient conditions to obtain the desired order of accuracy. In Section 3, we show the positivity-preserving property of the DG method as well as the bound-preserving technique for the mass fractions. In Section 4, we give numerical examples to demonstrate the performance of the scheme. Some concluding remarks will be given in Section 5.

\section{The ODE solver}
In this section, we demonstrate the time integrations to be used in this paper. We first review the MP RK method given in \cite{MPRK2,MPRK3} and then construct the MP MS methods.

\subsection{The modified Patankar Runge-Kutta methods}
In this subsection, we briefly present the MP RK methods proposed in \cite{MPRK2,MPRK3}. The general production-destruction system can be written as
\begin{eqnarray}
&& \frac{dc_i}{dt} = \sum_{j=1}^{N} p_{ij}(c)-\sum_{j=1}^{N} d_{ij}(c), \quad i = 1,2,\dots,N, \label{equation2.1}
\end{eqnarray}
where $c_i$ is the concentration of the $i$-th component ${\bf c}=(c_1,c_2,\cdots,c_M)^T$. %
The second order MP RK scheme \cite{MPRK2} is
\begin{subequations}
\begin{eqnarray}
c_{i}^{(0)} &=& c_{i}^n, \\
c_{i}^{(1)} &=&\alpha_{10} c_{i}^{(0)}+\Delta t\beta_{10}\left(\sum_j p_{ij}(c^{(0)})\frac{c_{j}^{(1)}}{c_{j}^{(0)}}-\sum_j d_{ij}(c^{(0)})\frac{c_{i}^{(1)}}{c_{i}^{(0)}}\right),\\
c_{i}^{n+1} &=&\alpha_{20}c_{i}^{(0)}+\alpha_{21}c_{i}^{(1)} \notag \\ 
&&+\Delta t \left(\sum_j\left(\beta_{20}p_{ij}(c^{(0)})+\beta_{21}p_{ij}(c^{(1)})\right)\frac{c_{j}^{n+1}}{(c_{j}^{(1)})^{s} (c_{j}^{(0)})^{1-s}}\right. \notag \\
&& \left.-\sum_j \left(\beta_{20}d_{ij}(c^{(0)})+\beta_{21}d_{ij}(c^{(1)})\right)\frac{c_{i}^{n+1}}{(c_{i}^{(1)})^{s} (c_{i}^{(0)})^{1-s}}\right) 
\end{eqnarray}
\end{subequations}
with the parameters satisfying 
$$
\alpha_{10}=1,\quad \alpha_{20}=1-\alpha, \quad \alpha_{21}=\alpha,
$$
$$
\beta_{10}=\beta, \quad \beta_{20}=1-\frac{1}{2\beta}-\alpha\beta, \quad \beta_{21}=\frac{1}{2\beta},
$$
$$
s=\frac{1-\alpha\beta+\alpha\beta^2}{\beta(1-\alpha\beta)}
$$
where
$$
0\leq\alpha\leq1, \quad \beta>0, \quad \alpha\beta+\frac{1}{2\beta}\leq1.
$$
The third order MP RK scheme \cite{MPRK3} is
\begin{subequations}
\begin{eqnarray}
c_{i}^{(0)} &=& c_{i}^n, \\
c_{i}^{(1)} &=&\alpha_{10} c_{i}^{(0)}+\Delta t\beta_{10}\left(\sum_j p_{i,j}(c^{(0)})\frac{c_{j}^{(1)}}{c_{j}^{(0)}}-\sum_j d_{i,j}(c^{(0)})\frac{c_{i}^{(1)}}{c_{i}^{(0)}}\right),\\
\rho_i &=& n_1c_i^{(1)}+n_2c_i^n\left(\frac{c_i^{(1)}}{c_i^n}\right)^2, \notag \\ 
c_i^{(2)} &=& \alpha_{20}c_i^{(0)}+\alpha_{21}c_i^{(1)} \notag \\
&&+\Delta t\sum_j\left(\left(\beta_{20}p_{ij}^{(0)}+\beta_{21}p_{ij}^{(1)}\right)\frac{c_j^{(2)}}{\rho_j}-\left(\beta_{20}d_{ij}^{(0)}+\beta_{21}d_{ij}^{(1)} \right)\frac{c_i^{(2)}}{\rho_i}\right), \\
\mu_i &=& c_i^n\left(\frac{c_i^{(1)}}{c_i^n}\right)^s, \notag \\
a_i &=& \eta_{1}c_i^n+\eta_{2}c_i^{(1)}+\Delta t\sum_j\left(\left(\eta_{3}p_{ij}^n+\eta_{4}p_{ij}^{(1)}\right)\frac{a_j}{\mu_j}-\left(\eta_{3}d_{ij}^n+\eta_{4}d_{ij}^{(1)}\right)\frac{a_i}{\mu_i}\right), \\
\sigma_i &=& a_i+zc_i^n\frac{c_i^{(2)}}{\rho_i}, \notag \\
c_{i}^{n+1} &=&\alpha_{30}c_{i}^{(0)}+\alpha_{31}c_{i}^{(1)}+\alpha_{32}c_{i}^{(2)} \notag \\ 
&&+\Delta t\sum_j\left(\left(\beta_{30}p_{ij}^{(0)}+\beta{31}p_{ij}^{(1)}+\beta_{32}p_{ij}^{(2)}\right)\frac{c_j^{n+1}}{\sigma_j}-\left(\beta_{30}d_{ij}^{(0)}+\beta_{31}d_{ij}^{(1)}+\beta{32}d_{ij}^{(2)}\right)\frac{c_i^{n+1}}{\sigma_i}\right) 
\end{eqnarray}
\end{subequations}
where the coefficients $\alpha_{ij}$ and $\beta_{ij}$, and the parameters $n_1$, $n_2$, $\eta_1$, $\eta_2$, $\eta_3$, $\eta_4$, $s$ and $z$ are given in \cite{MPRK3}.  
This scheme is third-order accurate and unconditionally positivity-preserving.

\subsection{The modified Patankar multi-step methods}
We focus on the production-destruction equations \eqref{equation2.1} and derive MP MS scheme in this subsection. In \cite{SSP}, the explicit SSP MS methods for the nonlinear ODE 
 \begin{eqnarray}
&& \frac{du}{dt} = L(u) 
\end{eqnarray}
are given as follows: for the second-order, 
 \begin{eqnarray}
&& u^{n+1} = \frac{1}{4} u^{n-2} +(\frac{3}{4} u^n + \frac{3}{2} \Delta t L(u^n)),
\end{eqnarray}
and for the third-order, 
 \begin{eqnarray}
&& u^{n+1} = (\frac{11}{27} u^{n-3} + \frac{4}{9} \Delta t L(u^{n-3})) + (\frac{16}{27} u^n + \frac{16}{9} \Delta t L(u^n)).
\end{eqnarray}
We can construct second and third order MP MS method according to it.

\subsubsection{Second-order scheme}
The second-order explicit SSP multistep method for \eqref{equation2.1} is
 \begin{eqnarray}
&& c_i^{n+1} = \frac{1}{4} c_i^{n-2} +\frac{3}{4} c_i^n + \frac{3}{2} \Delta t (\sum_{j} p_{ij}^n - \sum_{j} d_{ij}^n). \label{equation2.5}
\end{eqnarray}
Following \cite{MPRK2}, we make some modification to preserve conservation and positivity unconditionally. The MP MS method is rewritten as
 \begin{eqnarray}
&& c_i^{n+1} = \frac{1}{4} c_i^{n-2} +\frac{3}{4} c_i^n + \frac{3}{2} \Delta t (\sum_{j} p_{ij}^n \frac{c_j^{n+1}}{\sigma_j} - \sum_{j} d_{ij}^n \frac{c_i^{n+1}}{\sigma_i}), \label{equation2.6}
\end{eqnarray}
where $\sigma_i \geq 0$ is an undetermined function of $c_i^{n-2}$, $c_i^{n-1}$ and $c_i^{n}$.

If we assume
 \begin{eqnarray}
&& \sigma_i = c_i^{n+1} + O(\Delta t^2) \label{equation2.7}
\end{eqnarray}
then by \eqref{equation2.6} we have
\begin{align*}
c_i^{n+1} &= \frac{1}{4}c_i^{n-2} + \frac{3}{4}c_i^{n} + \frac{3}{2} \Delta t (\sum_{j}p_{ij}^n \frac{c_j^{n+1}}{c_j^{n+1}+O(\Delta t^2)} -\sum_{j}d_{ij}^n \frac{c_i^{n+1}}{c_i^{n+1}+O(\Delta t^2)})\\
&=\frac{1}{4}c_i^{n-2} + \frac{3}{4}c_i^{n} + \frac{3}{2} \Delta t (\sum_{j}p_{ij}^n (1+O(\Delta t^2)) -\sum_{j}d_{ij}^n (1+O(\Delta t^2)))\\
&=\frac{1}{4}c_i^{n-2} + \frac{3}{4}c_i^{n} + \frac{3}{2} \Delta t (\sum_{j}p_{ij}^n  -\sum_{j}d_{ij}^n ) +O(\Delta t^3)
\end{align*}
which shows that \eqref{equation2.6} is second-order accurate by using the fact that the explicit MS method \eqref{equation2.5} is second-order accurate. Therefore, \eqref{equation2.7} is a sufficient condition that \eqref{equation2.6} is second-order accurate.

In the second-order scheme \eqref{equation2.6}, if we take
 \begin{eqnarray}
&& \sigma_i = (c_i^n)^s(c_i^{n-1})^r(c_i^{n-2})^{1-r-s} \label{equation2.8}
\end{eqnarray}
with $s$ and $r$ to be determined later, then the Taylor expansions for $c_i^{n-k}$ at $t^{n+1}$ is
 \begin{eqnarray}
&& c_i^{n-k} = c_i^{n+1}-(k+1)\Delta t(P_i^{n+1}-D_i^{n+1})+O(\Delta t^2) \label{equation2.9}
\end{eqnarray}
for $k=0,1,2$. For convenience of notation, we denote 
 \begin{eqnarray}
&& Q_1 := P_i^{n+1}-D_i^{n+1}
\end{eqnarray}
for further analysis and then \eqref{equation2.9} becomes
 \begin{eqnarray}
&& c_i^{n-k} = c_i^{n+1}-(k+1)\Delta t Q_1 + O(\Delta t^2).
\end{eqnarray}
Therefore, the Taylor expansions for $\sigma_i$ at $t^{n+1}$ is
\begin{align*}
\sigma_i &= (c_i^{n})^s(c_i^{n-1})^r(c_i^{n-2})^{1-r-s}\\
&=(c_i^{n+1} -  \Delta t Q_1)^s(c_i^{n+1}-2\Delta t Q_1)^r(c_i^{n+1}-3\Delta t Q_1)^{1-r-s}+O(\Delta t^2)\\
&=c_i^{n+1} +(r+2s-3)\Delta t Q_1 +O(\Delta t^2),
\end{align*}
and now we have  
 \begin{eqnarray}
&& r+2s-3 = 0 \label{sufficient cond}
\end{eqnarray}
which is required for \eqref{equation2.8} to be satisfied.

Before stating the main theorem, we give the following lemma:
\begin{lemma}
{The MP MS scheme \eqref{equation2.6} is conservative. If $\sum_{i=1}^M c_i^{n-2}=\sum_{i=1}^M c_i^n=1$, then $\sum_{i=1}^M c_i^{n+1}=1$ for $i=1,2,...,M$.}
\end{lemma}
\begin{proof}
{We add up \eqref{equation2.6} over $i$ and use the fact that the modified Patankar scheme is conservative, i.e. $\sum_{i,j}( p_{ij}^n \frac{c_j^{n+1}}{\sigma_j} -  d_{ij}^n \frac{c_i^{n+1}}{\sigma_i})=0$ to obtain $$\sum_i c_i^{n+1}=\frac14 \sum_i c_i^{n-2} +\frac34 \sum_i c_i^n.$$}
\end{proof}
\begin{lemma}
The MP MS scheme \eqref{equation2.6} is unconditionally positivty-preserving. That is, for all $\Delta t >0$ and $c_i^{n-2}$,$c_i^{n-1}$ and $c_i^{n}>0$, we have $c_i^{n+1}>0$ for $i=1,2,...,M$.
\end{lemma}
The proof is similar to Lemma 2.7 in \cite{MPRK2_PP}, so we skip it. This lemma directly results in the following theorem.
\begin{theorem}
The MP MS scheme \eqref{equation2.6} is second-order accurate with
$$
\sigma_i = (c_i^n)^s(c_i^{n-1})^r(c_i^{n-2})^{1-r-s}
$$
and
$$
r+2s-3 = 0.
$$
Moreover, it is conservative in the sense that
$$
 \sum_{i=1}^M c_i^{n-2}=\sum_{i=1}^M c_i^{n-1}=\sum_{i=1}^M c_i^n=\sum_{i=1}^M c_i^{n+1}
$$
and unconditionally positivity-preserving:
if $c_i^n\geq0$ for $i=1,2,...M$, then $c_i^{n+1}\geq0$ for $i=1,2,...,M$.
\end{theorem}

\subsubsection{Third-order scheme}
Following the same approach in the previous subsection, we have the modified third order scheme
 \begin{eqnarray}
&& c_i^{n+1} = \frac{11}{27}c_i^{n-3} + \frac{16}{27}c_i^{n} + \Delta t(\sum_{j}(\frac{4}{9}p_{ij}^{n-3} + \frac{16}{9}p_{ij}^{n}) \frac{c_j^{n+1}}{\sigma_{j}} - (\frac{4}{9}d_{ij}^{n-3} + \frac{16}{9}d_{ij}^{n})\frac{c_i^{n+1}}{\sigma_{i}}), \label{equation2.13}
\end{eqnarray}
where $\sigma_i \geq 0$ are undetermined functions of $c_i^{n-k}$, $k = 0,1,2,3$. If we assume that 
 \begin{eqnarray}
&& \sigma_i = c_i^{n+1} + O(\Delta t^3), \label{equation2.14}
\end{eqnarray}
then \eqref{equation2.13} is third-order accurate.

We will derive explicit expressions of $\sigma_i$ in \eqref{equation2.13} for the sufficient condition \eqref{equation2.14} to be satisfied.
Following the analysis for the second-order scheme (2.7), we try to make
 \begin{eqnarray}
&& \sigma_i = (c_i^n)^s(c_i^{n-1})^r(c_i^{n-2})^p(c_i^{n-3})^{1-r-s-p}
\end{eqnarray}
with $s$, $r$ and $p$ to be determined. Taylor expansions for $c_i^{n-k}$ at $t^{n+1}$ with $k=0,1,2,3$ give
 \begin{eqnarray}
 c_i^{n-k} = c_i^{n+1} -(k+1)\Delta t(P_i^{n+1}-D_i^{n+1}) +\frac{(k+1)^2}{2}\Delta t^2\frac{\partial (P_i^{n+1}-D_i^{n+1})}{\partial c} (P^{n+1}-D^{n+1})
 +O(\Delta t^3). 
\end{eqnarray}
For convenience of notation, we denote
 \begin{eqnarray}
&& Q_1:= P_i^{n+1}-D_i^{n+1}, \quad Q_2:= \frac{\partial (P_i^{n+1}-D_i^{n+1})}{\partial c}(P^{n+1}-D^{n+1}),
\end{eqnarray}
and then
 \begin{eqnarray}
&& c_i^{n-k} = c_i^{n+1}-(k+1)\Delta tQ_1+\frac{(k+1)^2}{2}\Delta t^2 Q_2+O(\Delta t^3).
\end{eqnarray}
Taylor expansion for $\sigma_i$ at $t^{n+1}$ yields
\begin{align*}
\sigma_i &= (c_i^{n})^s(c_i^{n-1})^r(c_i^{n-2})^{p}(c_i^{n-3})^{1-r-s-p}\\
&=(c_i^{n+1} -  \Delta t Q_1 +\frac{1}{2}\Delta t^2 Q_2)^s(c_i^{n+1}-2\Delta t Q_1+2\Delta t^2 Q_2)^r \\
&\qquad (c_i^{n+1}-3\Delta t Q_1+\frac{9}{2} \Delta t^2 Q_2)^{p}(c_i^{n+1}-4\Delta t Q_1 +8\Delta t^2 Q_2)^{1-r-s-p}+O(\Delta t^3)\\
&=c_i^{n+1} +(p+2r+3s-4)\Delta t Q_1 -(7p+12r+15s-16)\frac{\Delta t^2}{2} Q_2\\
&\qquad +((p+2r+3s)^2-4r-9s-p)\frac{\Delta t^2 Q_1^2}{2c_i^{n+1}}+O(\Delta t^3)
\end{align*}
Letting the coefficients to be zero, we can solve out
\begin{eqnarray}
&& r=-3s+6, \quad p=3s-8.\label{sufficient cond1}
\end{eqnarray}

\begin{theorem}
The MP MS scheme \eqref{equation2.13} is third-order accurate with
$$
\sigma_i = (c_i^n)^s(c_i^{n-1})^r(c_i^{n-2})^p(c_i^{n-3})^{1-r-s-p}
$$
and
$$
r=-3s+6, \quad p=3s-8.
$$
Moreover, it is conservative which means that
$$
\sum_{i=1}^N c_i^{n-3}=\sum_{i=1}^N c_i^{n-2}=\sum_{i=1}^N c_i^{n-1}=\sum_{i=1}^N c_i^n=\sum_{i=1}^N c_i^{n+1}
$$
and unconditionally positivity-preserving:
if $c_i^n\geq0$ for $i=1,2,...M$, then $c_i^{n+1}\geq0$ for $i=1,2,...,M$.
\end{theorem}

\section{Bound-preserving DG scheme}
In this section, we discuss the positivity-preserving DG scheme. We first review the DG scheme and then discuss its positivity-preserving property. Then we demonstrate the bound preserving technique for mass fractions. Finally, we apply the ODE solver to our problem.

\subsection{The DG scheme}
In this subsection, we demonstrate the DG scheme for  \eqref{equation2}. We define the finite element space $V_h^k$ as
$$
V_h^k =\{v:v|_K\in Q^k(K),\:\:\forall K\in\Omega_h \},
$$
where $Q^k(K)$ is the set of tensor product polynomials of degree at most $k$ in cell $K$. Then the DG scheme is to find the numerical solution ${\bm{U_h}} \in [V_h^k]^{M+3}$ such that for all $v\in [V_h^k]^{M+3}$ we have 
\begin{eqnarray}
\frac{d}{dt}\int_K {\bm{U_h}}v d\bm{x}+ \sum_{e\in\partial K}\int_{e} \bm{H}(U_h^{int},U_n^{ext},\nu_e)v d\Gamma -\int_K { \bm{F}(\bm{U_h})} v_{x}d\bm{x}-\int_K { \bm{G}(\bm{U_h})} v_{y}d\bm{x}=\int_K { \bm{S}(\bm{U_h})}vd\bm{x}, \label{DG form}
\end{eqnarray}
where $\nu_e$ is the outward normal of the edge $e$ on element $K$. We use Lax-Friedrichs flux in this paper 
\begin{eqnarray}
\bm{H}(U_1,U_2,\bm{\nu}_e)=\frac{1}{2} [f(U_1) \bm{\nu}_e +f(U_2) \bm{\nu}_e- \alpha (U_1 -U_2)], \label{flux}
\end{eqnarray}
with $\alpha=\| | \langle u,v \rangle | +c\|_{\infty}$, where $c=\sqrt{\frac{\gamma p}{\rho}}$ is the sound speed. Note that  $f=( {\bm{F},\bm{G} } )$ and \\ $\bm{H}=(h_{\rho}, h_m,h_n,h_E,h_1,\cdots,h_{M-1})^T$. The following definition \cite{JDdetonation1,JDdetonation2} is used in the bound-preserving technique.
\begin{definition}
We say the elements in numerical flux $\bm{H}$ in \eqref{flux} are consistent if $h_{\rho}=h_i$ if we take $z_i=1$ for all $1\leq i\leq M-1$.
\end{definition}
As discussed in \cite{JDdetonation1,JDdetonation2}, the Lax-Friedrichs flux is consistent.
To present the positivity-preserving technique, we consider Euler forward time discretization. We take the test function to be 1, then the equation satisfied by the cell averages is
\begin{eqnarray}
\bar{{\bm{U}}}_K^{n+1}=\bar{U}_K^{n}-\frac{\Delta t}{|K|}\sum_{e\in\partial K}\int_e { \bm{H}(\bm{U_h}^{int},\bm{U_h}^{ext}},\bm{\nu}_e)d\Gamma +\frac{\Delta t}{|K|} \int_K {\bm{S}(\bm{U_h}})d\bm{x}, \label{cell average}
\end{eqnarray}
where $\bar{U}_K^n=\frac{1}{|K|}\int_K Ud\bm{x}$ is the cell average of $U$ in cell $K$ at time level $n$ and $\Delta t$ is the time step size. The integrals are approximated by proper quadrature rules which will be discussed in the next section.

\subsection{Bound-Preserving technique}
In this section, we develop the bound-preserving technique for the convection term. First, ignoring the source term from \eqref{cell average}, we have 
\begin{eqnarray}
\bar{\bm{U}}_K^{n+1}=\bar{\bm{U}}_K^{n}-\frac{\Delta t}{|K|}\sum_{e\in\partial K}\int_e \bm{H}(\bm{U_h}^{int},\bm{U_h}^{ext},\bm{\nu}_e)d\Gamma. \label{DG_convection}
\end{eqnarray}
We are seeking numerical approximations chosen from an admissible set $G$ defined as
$$G=\left\{ \bm{U} =\left(\begin{array}{cc}
\rho \\
m \\
n \\
E\\
c_1\\
...\\
c_{M-1}\\

\end{array}\right), \quad \rho>0, \quad p>0, \quad z_i>0, \quad i=1,...,M, \quad \sum_{i=1}^M z_i =1   \right\}.$$ 
As demonstrated in \cite{JDdetonation1}, $G$ is a convex set.

In this paper, we consider rectangular meshes. The spatial domain $\Omega=[a,b]\times [c,d]$ is partitioned into $ N\times N$ cells. The partitions are assumed to be uniform throughout this paper for simplicity. However, this assumption is not essential. We denote the cells in the $x$ and $y$ directions by 
$
I_i=[x_{i-1/2},x_{i+1/2}],
$
and
$J_j=[y_{j-1/2},y_{j+1/2}],
$
respectively.
for each cell $K \in \Omega$ and let $K_{ij}=[x_{i-\frac{1}{2}},x_{i+\frac{1}{2}}] \times [y_{j-\frac{1}{2}},y_{j+\frac{1}{2}}]$ be the $(i,j)$-th cell.  We then define a set of quadrature points $S_K$ on cell $K$. We use $L$ points  Gauss quadrature with $L\geq k+1$ for the integrals in \eqref{DG_convection}. The Gauss quadrature points on $[x_{i-\frac{1}{2}},x_{i+\frac{1}{2}}]$ and $[y_{j-\frac{1}{2}},y_{j+\frac{1}{2}}]$ are denoted as 
$$p_i^x=\{x_i^{\beta} :\beta = 1,...,L\} \quad and \quad p_j^y=\{y_j^{\beta} :\beta = 1,...,L\}.$$
In addition, we denote $\hat{L}$ Gauss-Lobatto points with $2\hat{L}-1\geq k$ on $[x_{i-\frac{1}{2}},x_{i+\frac{1}{2}}]$ and $[y_{j-\frac{1}{2}},y_{j+\frac{1}{2}}]$ as
$$\hat{p}_i^x=\{\hat{x}_i^{\alpha} :\alpha = 0,...,\hat{L}\} \quad and \quad \hat{p}_j^y=\{\hat{y}_j^{\alpha} :\alpha = 0,...,\hat{L}\}.$$
Define
\begin{eqnarray}
&& S_{K_{i,j}}=(p_i^x \otimes\hat{p}_j^y)\cup (\hat{p}_i^x \otimes p_j^y) \cup(p_i^x \otimes p_j^y). \label{S_k}
\end{eqnarray}
After defining the quadrature points, we can now state the theorem given in \cite{JDdetonation1,JDdetonation2} below.
\begin{theorem}
If the numerical approximation $\bm{U}_K(x,y) \in G \quad \forall (x,y)\in S_{K_{ij}}$, where $S_{K_{ij}}$ is defined in \eqref{S_k} for rectangular meshes, then the DG scheme \eqref{DG_convection} is positivity-preserving, namely, $\bar{\bm{U}}_K^{n+1} \in G$ under the time step restriction
\begin{eqnarray}
&& \alpha(\lambda_1+\lambda_2) \leq \hat{\omega}_1 \label{CFL}
\end{eqnarray}
where $\lambda_1 = \frac{\Delta {t}}{\Delta x}$ and $\lambda_2 = \frac{\Delta {t}}{\Delta y}.$ 
\end{theorem}

\subsection{Bound preserving technique for MP method}
In this subsection, we focus on the bound preserving technique for MP time integrations. We apply DG scheme as in section 3.1 to the spatial discretization except source term. We use $k$-point Gaussian quadrature for the integral in each direction in the convection terms and take Lagrangian basis so that the point values at the quadrature points are known. Therefore, we denote $c_{k,i,l}=c_{k,i,l}(t)$ as the point value at the $l$-th Gaussian quadrature point of the $i$-th species at the $k$-th cell after the spatial discretization, and define $c_{k,l}=(c_{k,1,l},\cdots,c_{k,M,l})^T$. To include the convection term and construct the semi-discrete scheme, we take the test function to be the $l$-th Lagrangian basis and formulate the system of ODEs base on the destruction and production equation \eqref{equation2.1} in the following form:
\begin{eqnarray}
&& \frac{dc_{k,i,l}}{dt}=F_{k,i,l}(c)+P_{k,i,l}(c)-D_{k,i,l}(c), \quad k=1,\cdots,N^2, \quad i=1,\cdots,M. \label{equation2.20}
\end{eqnarray}
 $F_{k,i,l}=F_{k,i,l}(c)$ denotes the contributions of the convection terms after spatial discretizations in the PDEs. The production and destruction terms are $P_{k,i,l}=P_{k,i,l}(c)=\sum_{j=1}^{M}p_{i,j}(c_{k,l})$ and $D_{k,i,l}=D_{k,i,l}(c)=\sum_{j=1}^{M} d_{i,j}(c_{k,l})$ which satisfy
$$p_{i,j}(c_{k,l})=d_{j,i}(c_{k,l})
\geq0, \quad \forall i,j,k,l \quad and \quad c\geq0$$
For simplicity of presentation, we drop the subscript $l$ in the rest of this section. We make the following assumption on \eqref{equation2.20} :\\
\textbf{Assumption 2.1.} The Euler forward method for the convection term satisfies the positivity-preserving property: if $c_{k,i}^n \geq 0$ for all $k,i$, then there exists $\Delta t_0>0$ such that
$$c_{k,i}^n+\Delta tF_{k,i}(c^n)\geq0$$
for all $k,i$ and  $\Delta t \leq \Delta t_0$.

The bound-preserving technique can be extended to SSP RK/MS methods which are convex combinations of forward Euler. For simplicity, we consider the second-order MPMS scheme \eqref{equation2.6} only and incorporate the convection term $F_{k,i}$. The fully discrete scheme is given as
\begin{eqnarray}
c_{k,i}^{n+1}=\frac{1}{4} c_{k,i}^{n-2} +\frac{3}{4} c_{k,i}^n +\frac{3}{2}\Delta t F_{k,i}(c^n)+\frac{3}{2}\Delta t\left(\sum_j p_{k,i,j}^n \frac{c_{k,j}^{n+1}}{\sigma_{k,j}}-\sum_{j}d_{k,i,j}^n \frac{c_{k,i}^{n+1}}{\sigma_{k,i}} \right) \label{MPMS2}
\end{eqnarray}
where
$$\sigma_i = (c_i^n)^s(c_i^{n-1})^r(c_i^{n-2})^{1-r-s}.$$
Clearly, the scheme is positivity-preserving if the time step satisfies
$$\Delta t \leq \frac{1}{2}\Delta t_0.$$

Similarly for the third-order MP MS scheme \eqref{equation2.6}, incorporating the convection term $F_{k,i}$ we have
\begin{eqnarray}
&& c_{k,i}^{n+1} = \frac{11}{27}c_{k,i}^{n-3} + \frac{16}{27}c_{k,i}^{n} + \Delta t(\frac{4}{9}F_{k,i}(c^{n-3})+\frac{16}{9}F_{k,i}(c^n)) \notag \\ \:\:\:
&& +\Delta t(\sum_{j}(\frac{4}{9}p_{k,i,j}^{n-3} + \frac{16}{9}p_{k,i,j}^{n}) \frac{c_{k,j}^{n+1}}{\sigma_{j}} - (\frac{4}{9}d_{k,i,j}^{n-3} + \frac{16}{9}d_{k,i,j}^{n})\frac{c_{k,i}^{n+1}}{\sigma_{i}}) \label{MPMS3}
\end{eqnarray}
where
$$\sigma_i = (c_i^n)^s(c_i^{n-1})^r(c_i^{n-2})^{p}(c_i^{n-3})^{1-r-s-p}.$$
The scheme is positivity-preserving if the time step satisfies $$\Delta t \leq \frac{1}{3}\Delta t_0.$$

\subsection{The limiter}
In this subsection, We discuss the limiter to be used to deal with negative numerical approximations with positive numerical cell averages.

To enforce the positivity, a suitable limiter {\cite{limiter}} can be added to keep density, pressure and mass fraction being within the physical bounds. Following \cite{JDdetonation1,JDdetonation2}, the main idea of the positivity-preserving limiter is to modify the DG polynomial $w_K(x)$ into polynomial
$$
\tilde{w}_K=\theta_K (w_K-\bar{w}_K)+\bar{w}_k
$$
where $\bar{w}_K\in G$ is the cell average and $\theta_K\in [0,1]$. {This limiter does not degenerate the accuracy of the DG polynomial $w_K(x)$ \cite{limiter_proof}.}\\
The algorithm of the limiter on each fixed element $K$ is given below:
\begin{enumerate}
\item For a small number $\epsilon = 10^{-13}$, if $\bar{\rho} > \epsilon$, we proceed to the next step. Otherwise, we simply take $U=\bar{U}$
\item Enforce the positivity of the density $\rho$: Compute the minimum value as
$$ \rho_{min} = \min_{(x,y)\in S_{K}} \rho(x,y), $$ where $S_K$ is defined in \eqref{S_k}.
If $\rho_{min} < 0$, then take
$$\hat{\rho}=\bar{\rho}+\theta_K (\rho-\bar{\rho}),$$
$$\hat{c}_i=\bar{c}_i+\theta_K (c_i -\bar{c}_i), i=1,...,M-1$$
with
$$\theta_K = \frac{\bar{\rho}-\epsilon}{\bar{\rho}-\rho_{min}}.$$
{Here $\hat{r}_M=\bar{r}_M+\theta(r_M -\bar{r}_M)$ to keep $\sum_i^M  \hat{r}_i = \hat{\rho}$. }
\item Enforce positivity of the mass fractions:
For $1\leq i \leq M$, define $\hat{S}_i=\{(x,y)\in S_K: \hat{c}_i(x,y) \leq 0\}$. Take
\begin{eqnarray}
\tilde{c}_i = \hat{c}_i +\theta (\frac{\bar{c}_i}{\bar{\rho}} \hat{\rho}-\hat{c}_i),\ 1\leq i\leq M-1, \notag \\ \:\:\:
\theta=\max_{1\leq i\leq M} \max_{(x,y)\in \hat{S}_i} \{\frac{-\hat{c}_i(x,y) \bar{\rho}}{\bar{c}_i \hat{\rho}(x,y)-\hat{c}_i(x,y) \bar{\rho}},0\}.
\end{eqnarray}

\item Modify the pressure:
Denote $\tilde{{\bm{U}}}=( \hat{\rho},m ,n,E,\tilde{c}_1,...,\tilde{c}_{M-1} )^T$. For each $\bm{x} \in S$, if $\tilde{\bm{U}} \in G$, then take $\theta _{\bm{x}}=1$. Otherwise, take
$$\theta_{\bm{x}}=\frac{p(\bar{\bm{U}})}{p(\bar{\bm{U}})-p(\tilde{\bm{U}}(x))}. $$
Then, we use
$$\bm{U}^{new} = \bar{\bm{U}} +\theta(\tilde{\bm{U}}- \bar{\bm{U}}), \quad \theta=\min_{\bm{x} \in S_K} \theta_{\bm{x}}, $$
as the new DG approximation. 
\end{enumerate}

\subsection{Full algorithm and the main theorem}
We have discussed the DG scheme and MP MS time integration. Putting them together, we have a conservative and positivity-preserving scheme for solving Euler equations.
 {Now we demonstrate the full algorithm of our method in the flow chart. We use the MPMS2 method on rectangular meshes as an example:}
\begin{enumerate}
    \item Consider the convection term only
    $$
    {\bm{U}}_{k,i}^{(1)}=\frac{1}{4} \bm{U}_{k,i}^{n-2} +\frac{3}{4} \bm{U}_{k,i}^n +\frac{3}{2}\Delta t \bm{F}_{k,i}(\bm{U}^n),
    $$
    where $\bm{U}=(\rho z_1,\cdots,\rho z_M,m,E)^T.$
    \item Apply limiter {in Section 3.4} to polynomials at the Gaussian points with $S_{k_{1}}=p_i^x \otimes  p_j^y$.
    \item Use Patankar to deal with source 
    $$
    c_{k,i}^{n+1}=c_{k,i}^{(1)}+\frac{3}{2}\Delta t\left(\sum_j p_{k,i,j}^n \frac{c_{k,j}^{n+1}}{\sigma_{k,j}}-\sum_{j}d_{k,i,j}^n \frac{c_{k,i}^{n+1}}{\sigma_{k,i}} \right),
    $$
    where $c=(\rho z_1,\cdots,\rho z_M)^T.$
    
    \item Apply limiter {in Section 3.4} to polynomials at points in $S_{k_{2}}=(\hat{p}_i^x \otimes p_j^y) \cup(p_i^x \otimes \hat{p}_j^y)$.
    \item Let $\bm{U}^{n-2}=\bm{U}^{n-1}$, $\bm{U}^{n-1}=\bm{U}^n$, and $\bm{U}^n=\bm{U}^{n+1}$, then
    go back to step 1 and repeat the process 1-4 until the final time is reached.
\end{enumerate}
\noindent
{\textbf{Remark 3.1.} The reason we apply the limiter twice in the above flow chart is because each limiter aims for different set of points. In Step 2, the limiter is applied on the quadrature points. In this way, we ensure that the point values on quadrature points are kept positive in each cell $K$. Fortunately, Patankar in Step 3 will not destroy the positivity of the cell average and the quadrature point values, but it may not yield positive numerical approximations on the cell interfaces. So we apply the limiter again in Step 4 to those points. The numerical approximations are guaranteed to be in the admissible set.}
\\
\noindent
The main theorem is stated below.

\begin{theorem}
Consider the DG scheme \eqref{DG form} coupled with the second-order MP MS method \eqref{MPMS2} or the third-order one \eqref{MPMS3}.The schemes are bound preserving: if $\bm{U}_{k,i}^n \in G$ for all $1\leq k\leq N^2$,$1\leq i\leq M$, then $\bm{U}_{k,i}^{n+1} \in G$ under the condition $\Delta t \leq \min_i \frac{\alpha_i}{\beta_i}\Delta \tilde{t}$ with $\Delta \tilde{t}$ satisfies the condition in \eqref{CFL}. 
\end{theorem}
\begin{proof}
Let us first prove the positivity preserving property for the DG scheme with the first order explicit Euler forward time integration. Without source term, the equation satisfied by cell average is given in \eqref{cell average}. The positivity-preserving technique yields positive numerical cell average. Then applying slope limiter, we obtain the positive numerical approximation at Gaussian points. The momentum and energy equations are trivial since there is no source for them. For the equations of mass fractions, we approximate the source terms using MP trick
\begin{eqnarray}
\sum_{j} p_{ij}^n \frac{c_j^{n+1}}{\sigma_j} - \sum_{j} d_{ij}^n \frac{c_i^{n+1}}{\sigma_i}. \label{P-D}
\end{eqnarray}
The goal is to show that $\bar{c}_k^{n+1}\geq0$. A sufficient condition is to obtain positive numerical approximations at the Gaussian quadrature points and this can be achieved by the MP time integration. We can refer to \cite{modified Patankar} for the proof of positivity of the technique.
Extend to second and third order multi-step time discretization will keep the positivity since they are convex combinations of Euler forward scheme.

{To start the discussion about the upper bound of the mass fraction $z_i\leq 1$, we first subtract \eqref{1.1d}-\eqref{1.1e} from \eqref{1.1a} to obtain a new equation
\begin{eqnarray}
(c_M)_t+(mz_M)_x+(nz_M)_y=s_M \label{ghost eqn}
\end{eqnarray}}
Then we apply the positivity-preserving technique to each $c_{k,i}^n$ {for $i=1,\dots,M$}.
Therefore, $c_i>0 (z_i>0)$ are obtained at all time levels. Moreover, we use consistent fluxes and conservative time integration to preserve the total mass $\sum_{i=1}^M c_i=\rho$. As a result, the numerical approximation of $z_i$ is bounded in interval $[0,1]$.
\end{proof}

\section{Numerical examples} In this section, we provide numerical experiments to show the performance of the numerical scheme. %

\subsection{Test of the ODE solver}
In this part, we test the accuracy of MP MS methods in solving linear and nonlinear ODEs
\begin{ex}\label{ex1}
This example is to test the accuracy on a linear problem:\\
$$
\frac{dc_1}{dt}=c_2-ac_1,
$$
$$
\frac{dc_2}{dt}=ac_1-c_2,
$$
with constant $a>0$, and initial value $c_1(0)=c_1^0$, $c_2(0)=c_2^0$. The exact solutions are 
$$
c_1(t)=(1+b\exp(-(a+1)t))c_1^\infty, \quad c_2(t)=c_1^0+c_2^0-c_1(t),
$$
with the parameters $c_1^\infty$ and $b$ determined by
$$
c_1^\infty=\frac{c_1^0+c_2^0}{a+1}, \quad b=\frac{c_1^0}{c_1^\infty}-1.
$$
In the numerical experiment, we take $c_1^0=4.5$, $c_2^0=3.2$, $a=2.7$ and the final time $t=1$.
\end{ex}

\begin{table}[!htbp]
\begin{center}
\begin{tabular}{c|c|c|c|c}
\hline
&\multicolumn{2}{c|}{MPMS2 }&\multicolumn{2}{c}{MPMS3 }   \\
\hline
$\Delta t$ &  Error & Order&  Error & Order\\
\hline
1/20  &1.88e-02& -- &1.03e-03& --  \\
1/40 & 4.56e-03&2.04&1.27e-04&3.02\\
1/80 &1.09e-03&2.07&1.58e-05&3.00\\
1/160 &2.75e-04&1.99&2.04e-06&2.96\\
1/320 &6.97e-05&1.98&2.61e-07&2.97\\
\hline
\end{tabular}
\caption{\label{table:1}Example \ref{ex1}: Accuracy test of MPMS2 and MPMS3 for linear ODEs}
\end{center}
\end{table}
The errors between the numerical and exact solutions at the final time are listed in Table \ref{table:1}. From the table, we can observe optimal convergence rates for second and third order MP MS methods.

\begin{ex}\label{ex2}
This example is to test the accuracy for solving a nonlinear problem given in \cite{MPRK2}:\\
\begin{eqnarray*}
\frac{dc_1}{dt}&=&F_1(c)-\frac{c_1c_2}{c_1+1},\\
\frac{dc_2}{dt}&=&F_2(c)+\frac{c_1c_2}{c_1+1}-ac_2,\\
\frac{dc_3}{dt}&=&F_3(c)+ac_2,
\end{eqnarray*}
where $(F_1(c), F_2(c), F_3(c))$ denotes "convection terms".

To express this system of ODEs in the form of production-destruction equations, we set
$$
p_{21}=d_{12}(c)=\frac{c_1 c_2}{c_1+1}, \quad p_{32}(c)=d_{23}(c)=ac_2,
$$
and $p_{ij}=d_{ij}=1$ for other sets of $i$, $j$. The initial conditions are set as $c_1^0=9.98$, $c_2^0=0.01$ and $c_3^0=0.01$. The convection terms are
$$
(F_1(c), F_2(c), F_3(c))=\left(c_1c_2c_3, \frac{c_3}{c_2}, c_2c_2c_3^2\right).
$$
The final time $t=1$ and the parameter $a=1$.
\end{ex}
The errors between the numerical and exact solutions at the final time are listed in Table \ref{table:2}, from which we can also observe the optimal convergence rates.

\begin{table}[!htbp]
\begin{center}
\begin{tabular}{c|c|c|c|c}
\hline
&\multicolumn{2}{c|}{MPMS2 }&\multicolumn{2}{c}{MPMS3 }   \\
\hline
$\Delta t$ &  Error & Order&  Error & Order\\
\hline
1/20  &2.13e-03& -- &3.67e-04& --  \\
1/40 &4.43e-04&2.26&3.68e-05&3.32\\
1/80 &1.14e-04&1.96&4.115e-06&3.16\\
1/160 &2.72e-05&2.06&4.87e-07&3.07\\
1/320 &6.76e-06&2.01&5.924e-08&3.04\\
\hline
\end{tabular}
\caption{\label{table:2}Example \ref{ex2}: Accuracy test of MPMS2 and MPMS3 for nonlinear ODEs}
\end{center}
\end{table}

\noindent
{\textbf{Remark 4.1.} The parameters $s$, $r$, and $p$ were chosen based on the relationships in \eqref{sufficient cond} and \eqref{sufficient cond1} for the numerical schemes to be at their desired accuracy. We tested for different values of $s$ from $-1$ to $5$ hence different values of $r$ and $p$ with fixed time step in example 4.1 and example 4.2. We want to know the influence of different set of parameters on the error. The results are the four plots given in Figure \ref{parameter}. At this point, we could only observe and conclude that there seems to be a set of parameters which gives the smallest error. This is an interesting topic and we leave it in the future work.}
\\
\begin{figure}[!htbp]
\subfigure[Example 4.1 using MPMS2]{\includegraphics[scale=0.215]{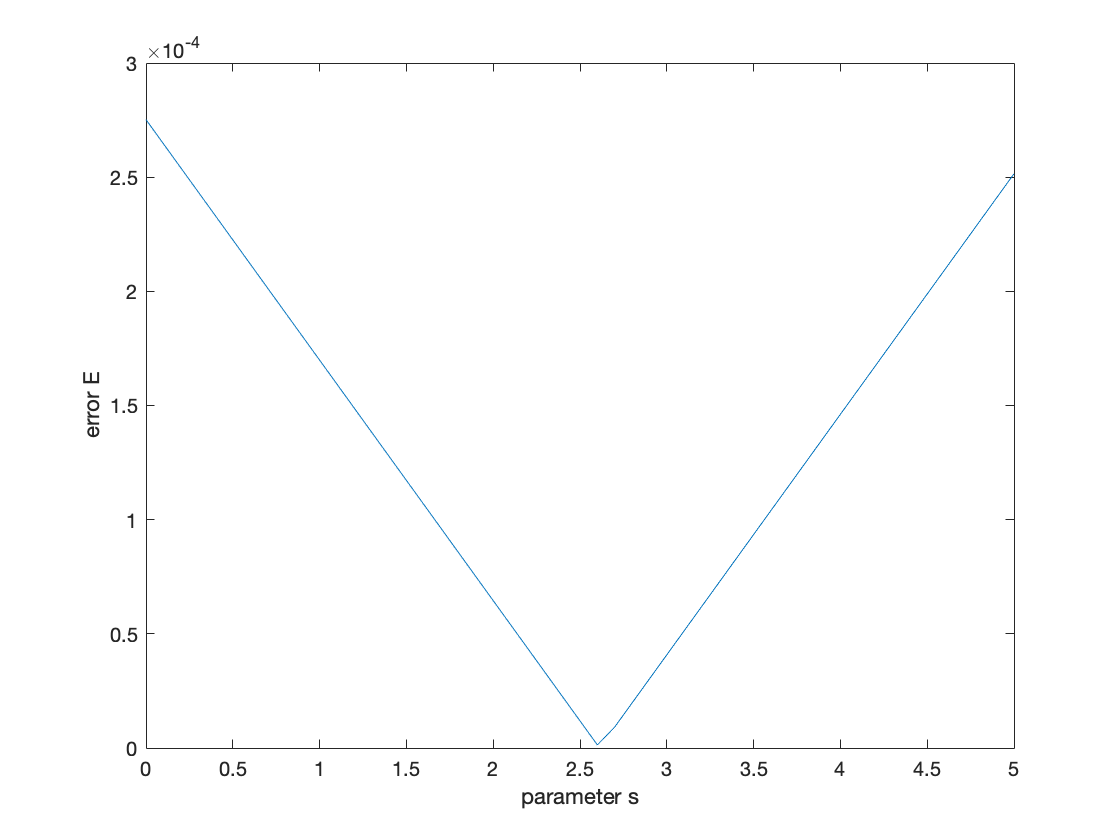}}
\subfigure[Example 4.1 using MPMS3]{\includegraphics[scale=0.215]{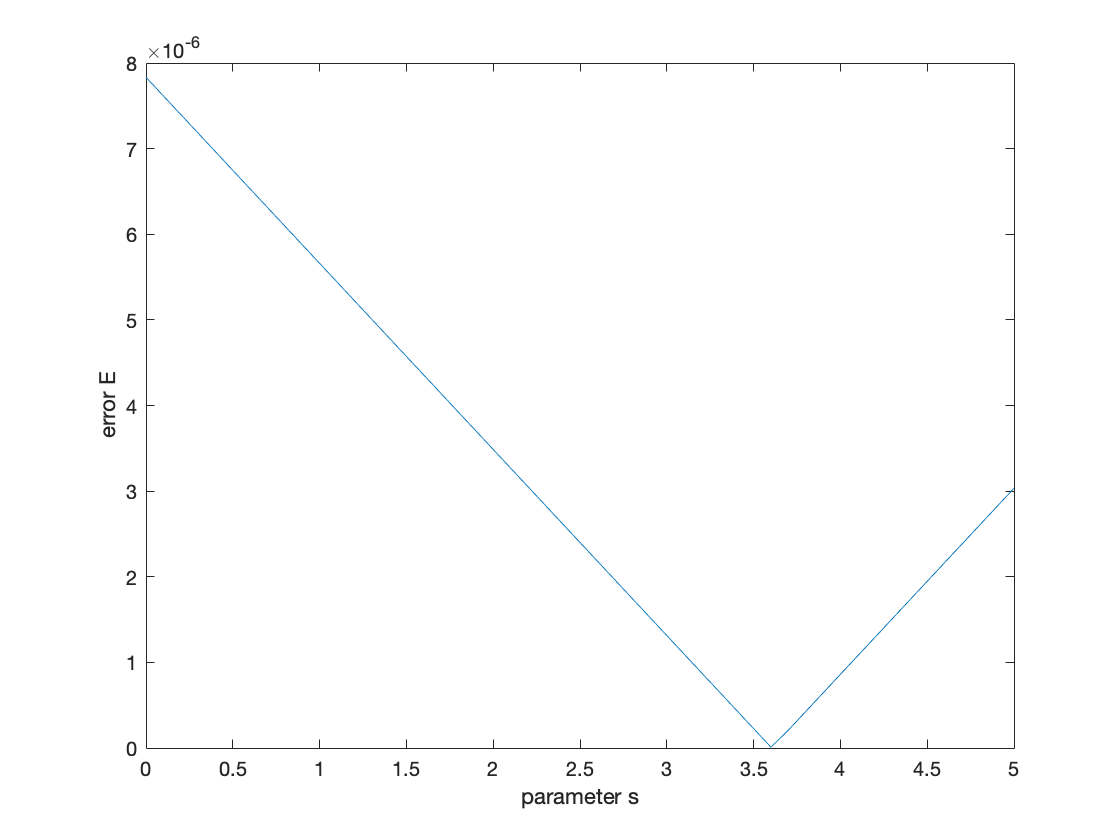}}\\
\subfigure[Example 4.2 using MPMS2]{\includegraphics[scale=0.215]{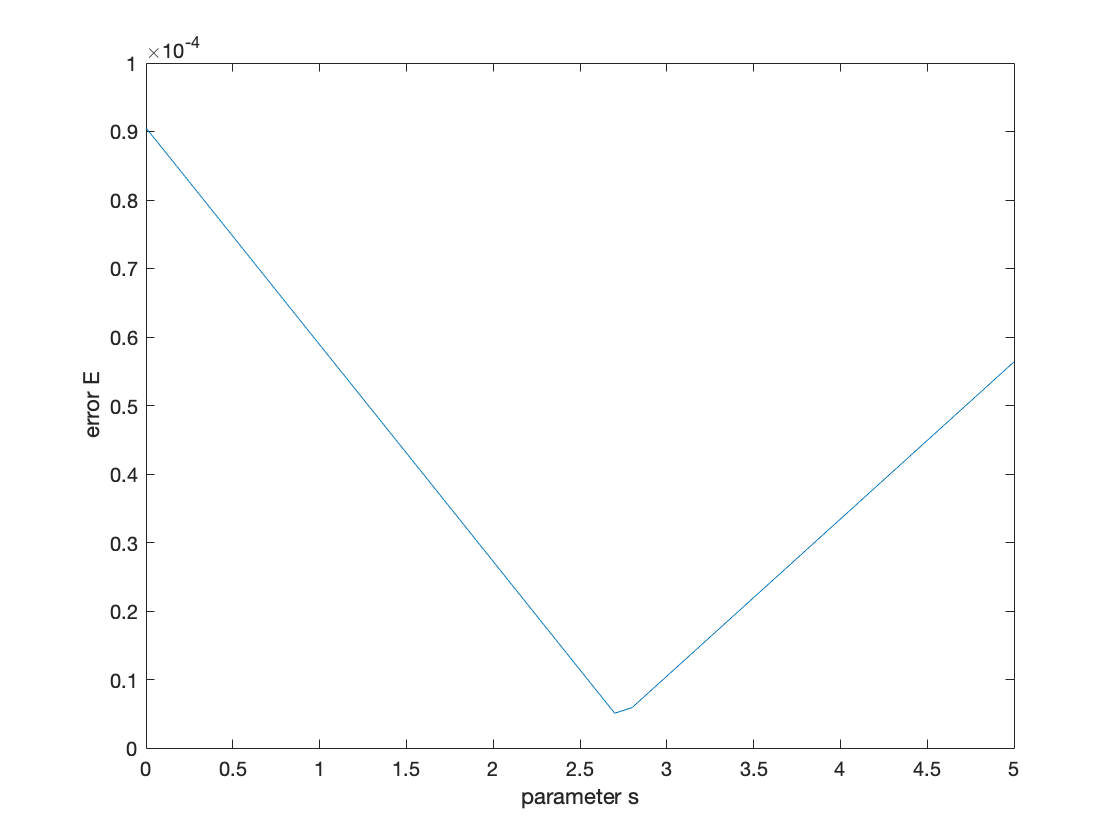}}
\subfigure[Example 4.2 using MPMS3]{\includegraphics[scale=0.215]{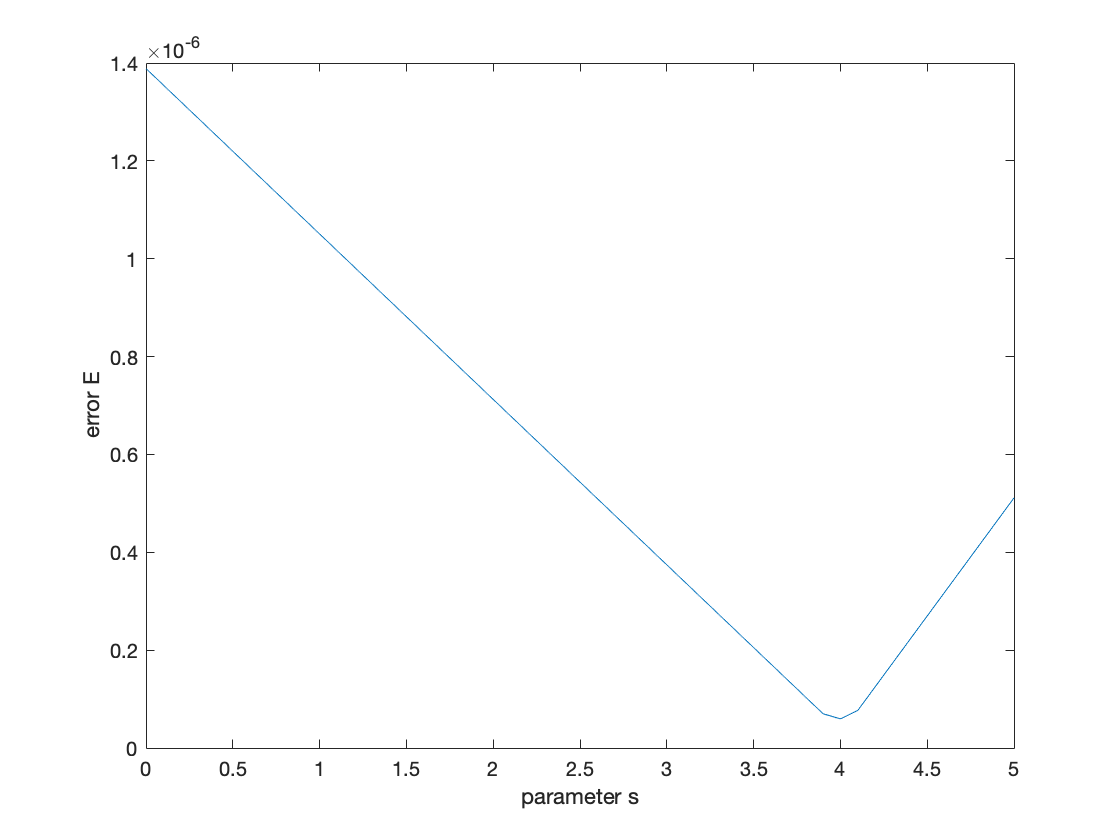}}\\
\caption{Parameter $s$ and error plots}
\label{parameter}
\end{figure} 

\subsection{Euler Equations with Three species Reactions and General Equation of State}
\begin{ex} \label{ex3}
We solve the three species model of the one-dimensional Euler system with a more general equation of state in \cite{1D example} 
\begin{eqnarray*}
U&=&(\rho_1, \rho_2, \rho_3, \rho u, E)^T,\\
F(U)&=&(\rho_1 u, \rho_2 u, \rho_3 u, \rho u^2+P, (E+p)u)^T,\\
S(U)&=&(2M_1\omega, -M_2\omega,0 ,0, 0)^T.
\end{eqnarray*}
The rate of the chemical reaction is given by
$$
\omega=\left(k_f(T)\frac{\rho_2}{M_2}-k_b(T)\left(\frac{\rho_1}{M_1}\right)^2\right)\sum_{s=1}^3 \frac{\rho_s}{M_s}, \quad k_f=CT^{-2}e^{-E/T},
$$
$$
k_b=k_f /exp(b_1+b_2logz +b_3z + b_4 z^2 + b_5z^3),\quad z=10000/T.
$$
The parameters are $M_1=0.016$, $M_2=0.032$, $M_3=0.028$, $q_1=1.558 \times 10^7$, $R=8.31447215$, $C_0=2.9\times10^{17} m^3$, $E_0=59750 K$, and $b_1 = 2.855$, $b_2=0.988$, $b_3=-6.181$, $b_4=-0.023$, $b_5 = -0.001$. For this model, we split the source terms $2M_1\omega$ into two parts:
$$
2M_1\omega=\omega_+ - \omega_-,
$$
with 
$$
\omega_+=2M_1k_f(T)\frac{\rho_2}{M_2}\sum_{s=1}^3\frac{\rho_s}{M_s}\geq0 \quad and \quad \omega_-=2M_1k_b(T)\left(\frac{\rho_1}{M_1}\right)^2\sum_{s=1}^3\frac{\rho_s}{M_s}\geq0.
$$
The production and destruction terms are
$$
p_{2,1}=d_{1,2}=\omega_-\geq0, \quad p_{1,2}=d_{2,1}=\omega_+\geq0.
$$
The eigenvalues of the Jacobian are $(u,u,u,u-c,u+c)$ where $c=\sqrt{\frac{\gamma p}{\rho}}$ with $\gamma=1+\frac{p}{T\sum_{s=1}^3 \rho_s e_s^{'}(T)}$. The initial conditions are: the densities $\rho_1$,$\rho_2$ and $\rho_3$ are $5.251896311257204 \times 10^{-5}$, $3.748071704863518\times 10^{-5}$,$2.962489471973072\times 10^{-4}$ on the left, and $8.341661837019181 \times 10^{-8}$, $9.455418692098664\times 10^{-11}$, $2.748909430004963\times 10^{-7}$ on the right. The velocities are zero. The pressures are 1000 on the left and 1 on the right. The final time is $t=0.0001$. The densities, velocity and pressure are presented in Figure\ref{4.3_MPRK2} for MPRK2, Figure \ref{4.3_MPMS2} for MPMS2 and Figure \ref{4.3_MPMS3} for MPMS3 methods. We can observe some oscillations in the numerical approximations. This is because we did not apply oscillation suppressors such as WENO algorithm in the numerical scheme. Though the oscillations exist, the bound-preserving technique is enough to stabilize our numerical scheme.

\end{ex}

\begin{figure}[!htbp]
\subfigure[$ \rho_1$]{\includegraphics[width = 3in]{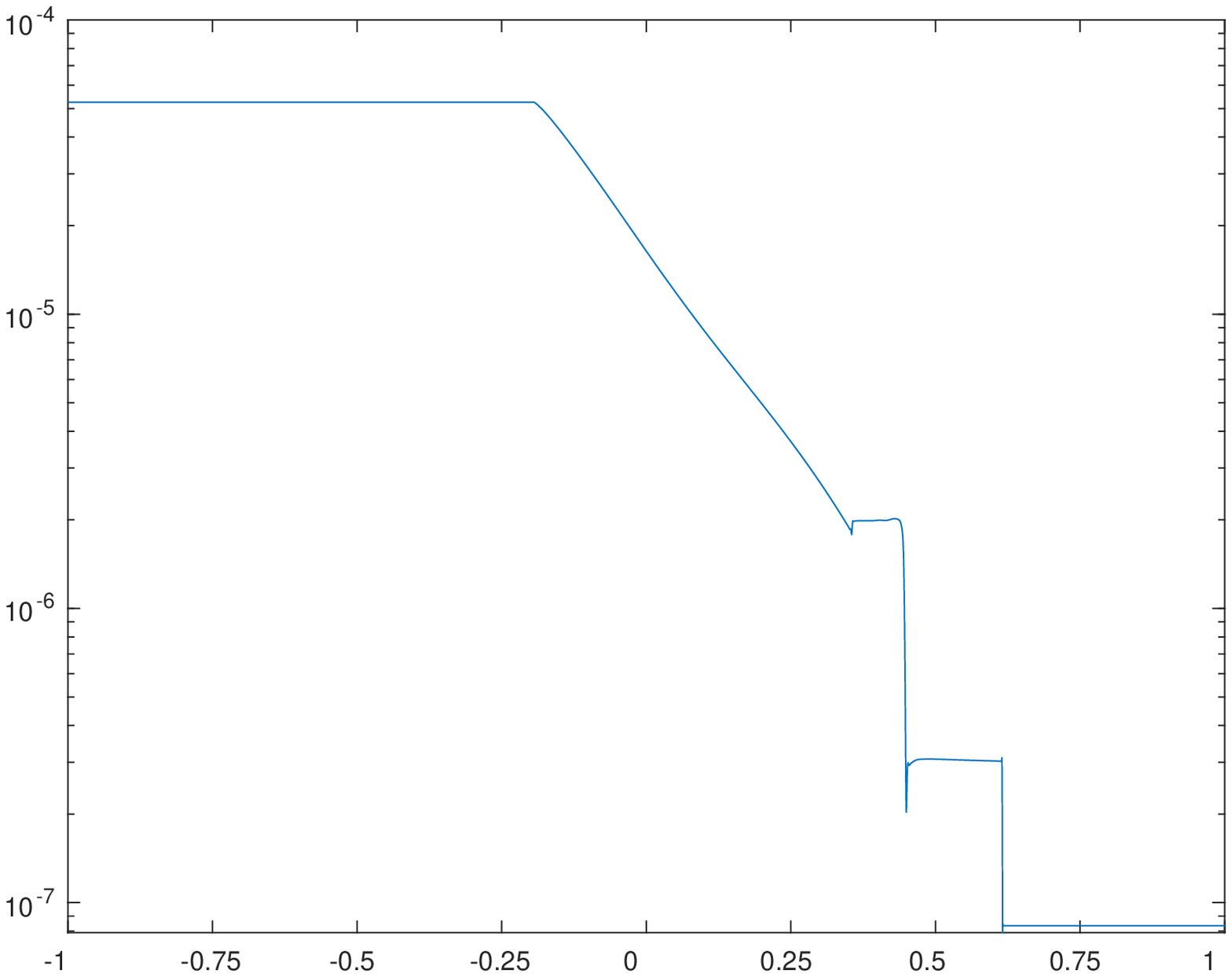}}
\subfigure[$ \rho_2$]{\includegraphics[width = 3in]{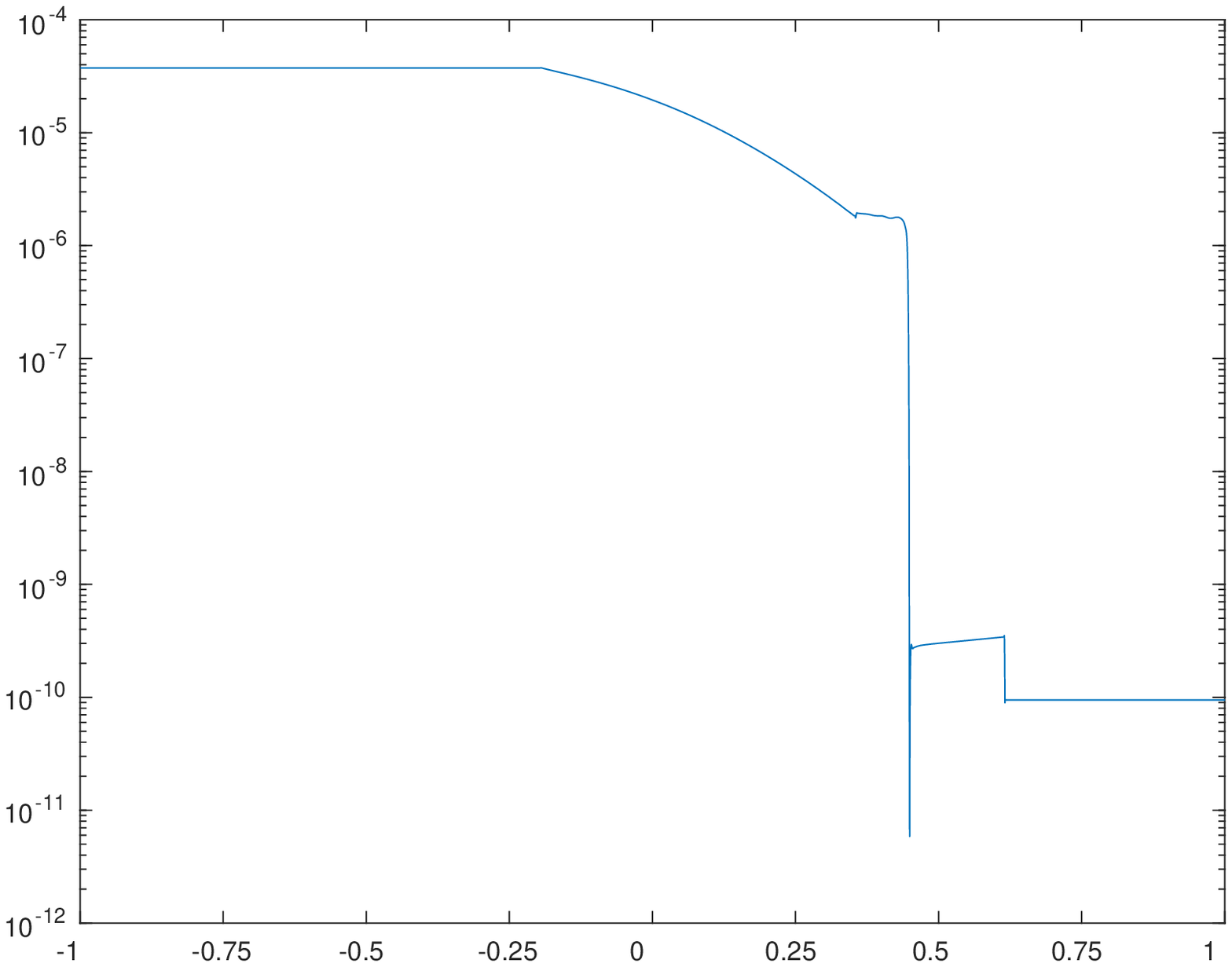}}\\
\subfigure[$ \rho_3$]{\includegraphics[width = 3in]{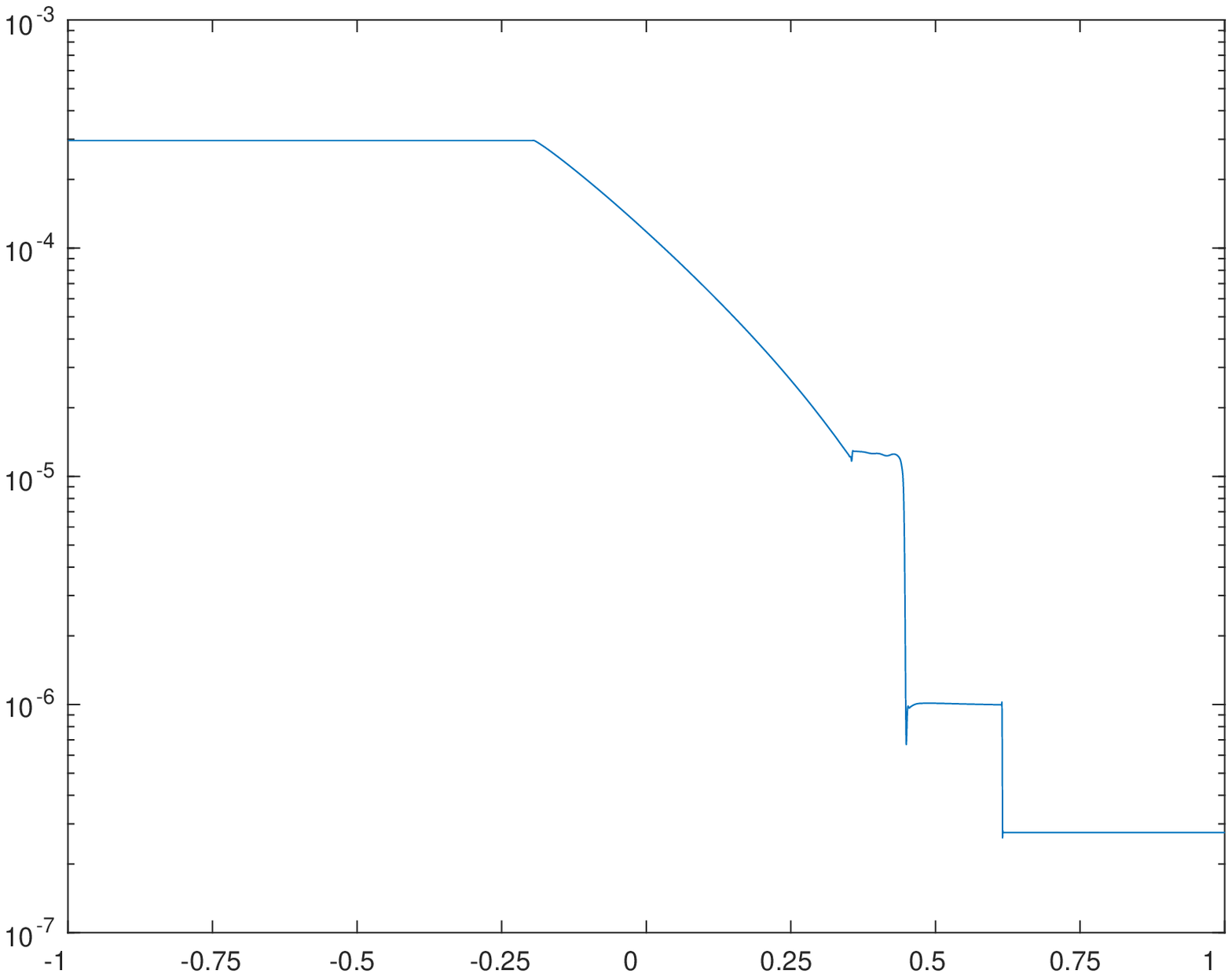}}
\subfigure[$ u$]{\includegraphics[width = 3in]{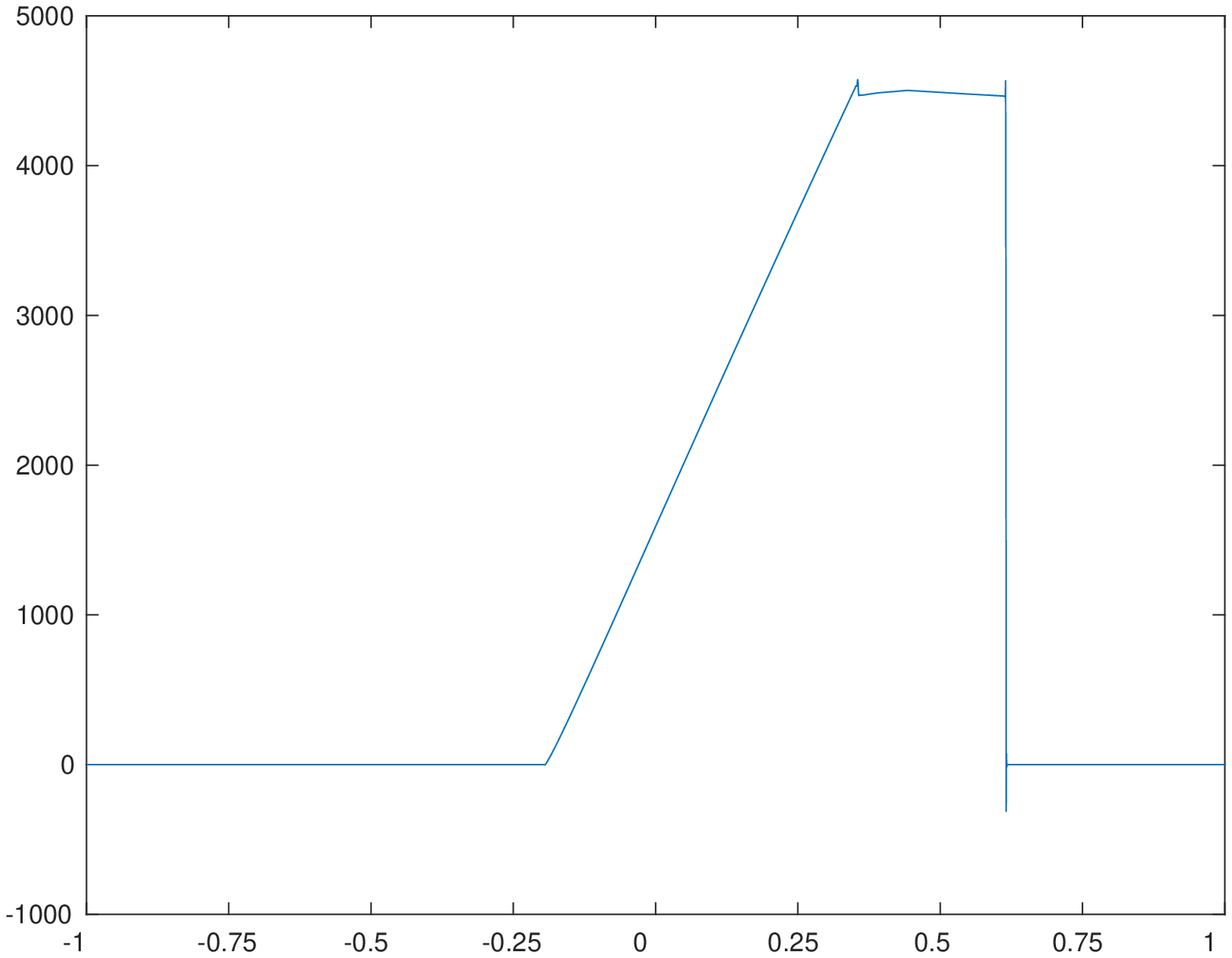}}\\
\subfigure[$ p$]{\includegraphics[width = 3in]{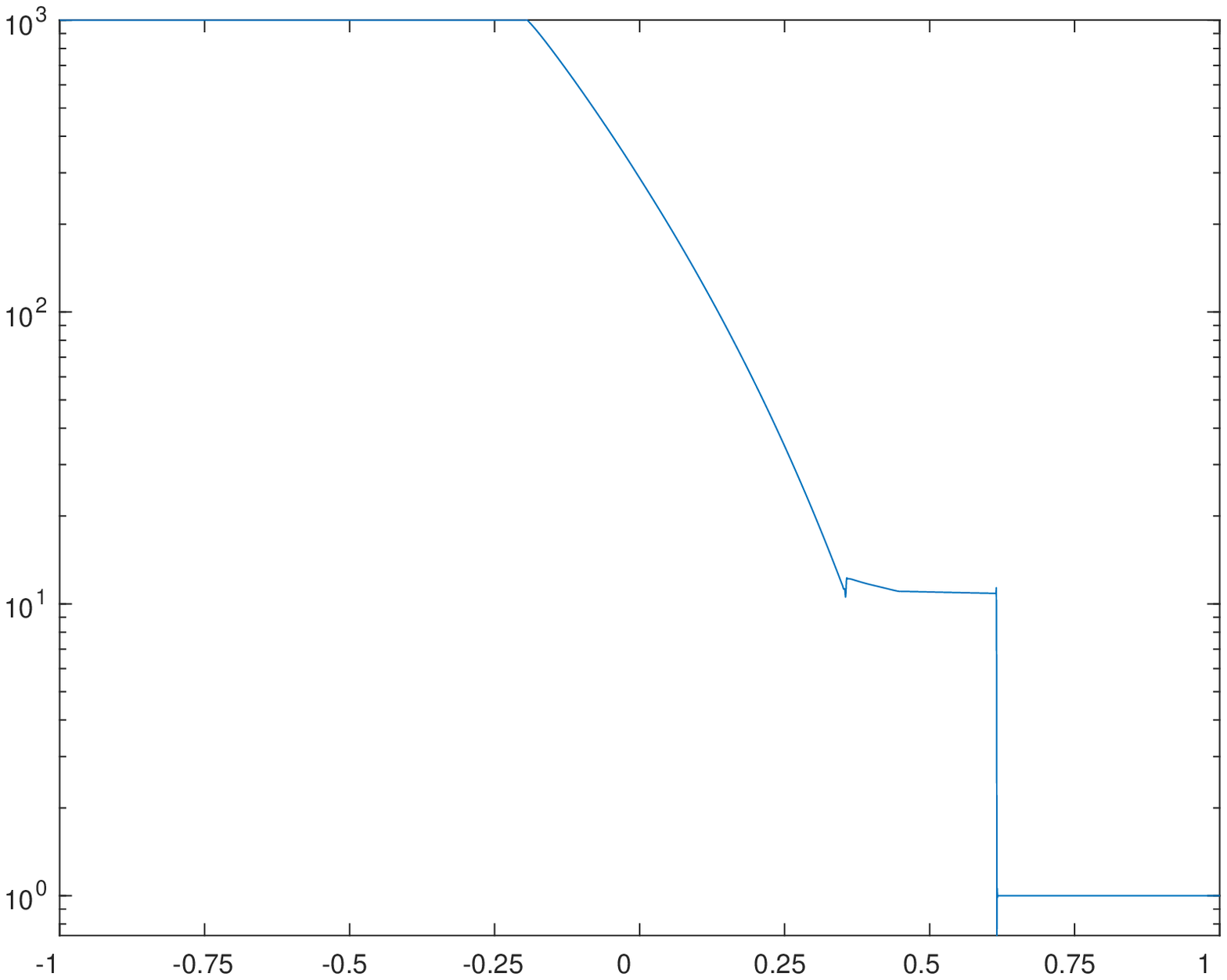}}
\caption{Example 4.3: three species reaction problem at t=0.0001 using MPRK2}
\label{4.3_MPRK2}
\end{figure}

\begin{figure}[!htbp]
\subfigure[$ \rho_1$]{\includegraphics[width = 3in]{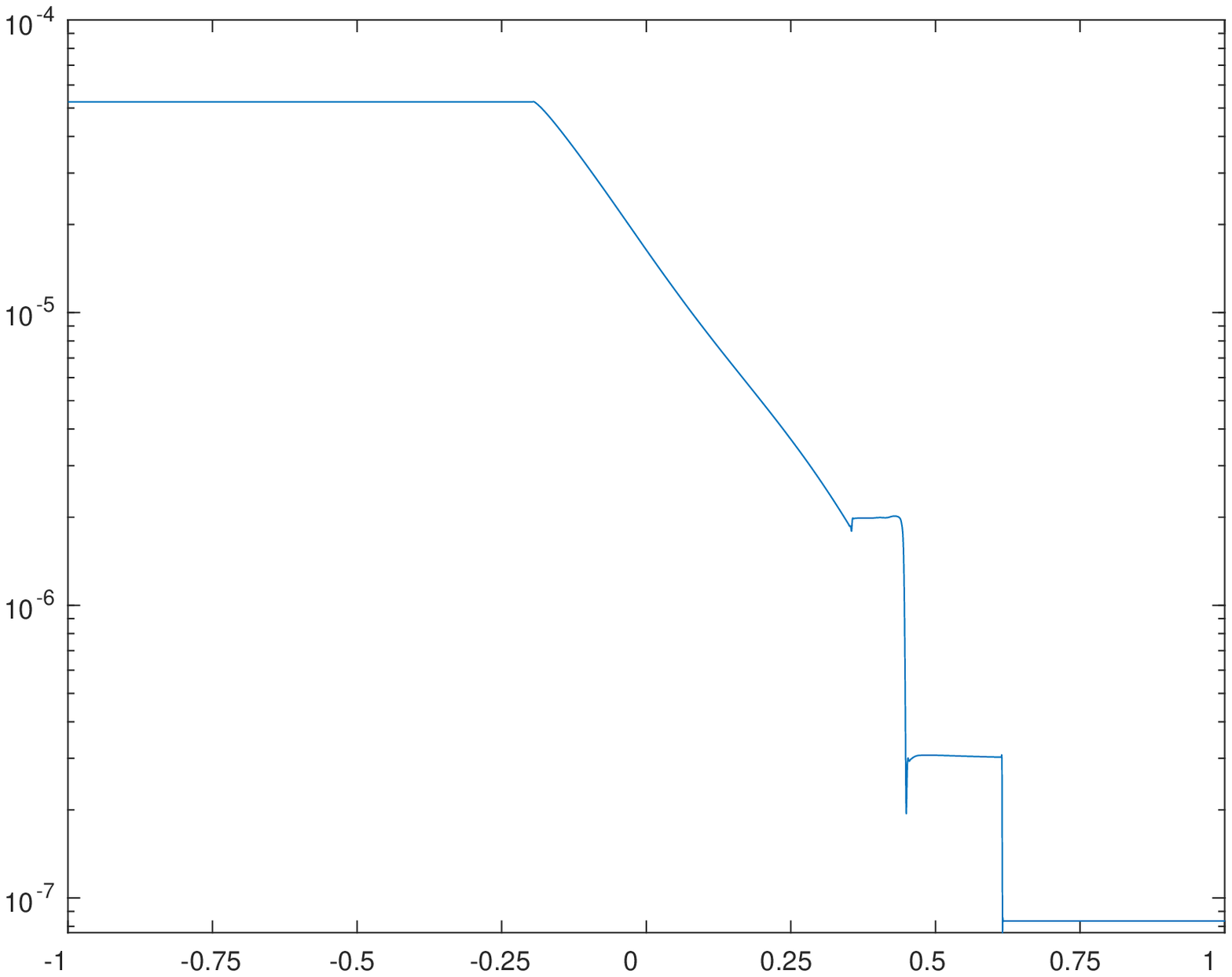}}
\subfigure[$ \rho_2$]{\includegraphics[width = 3in]{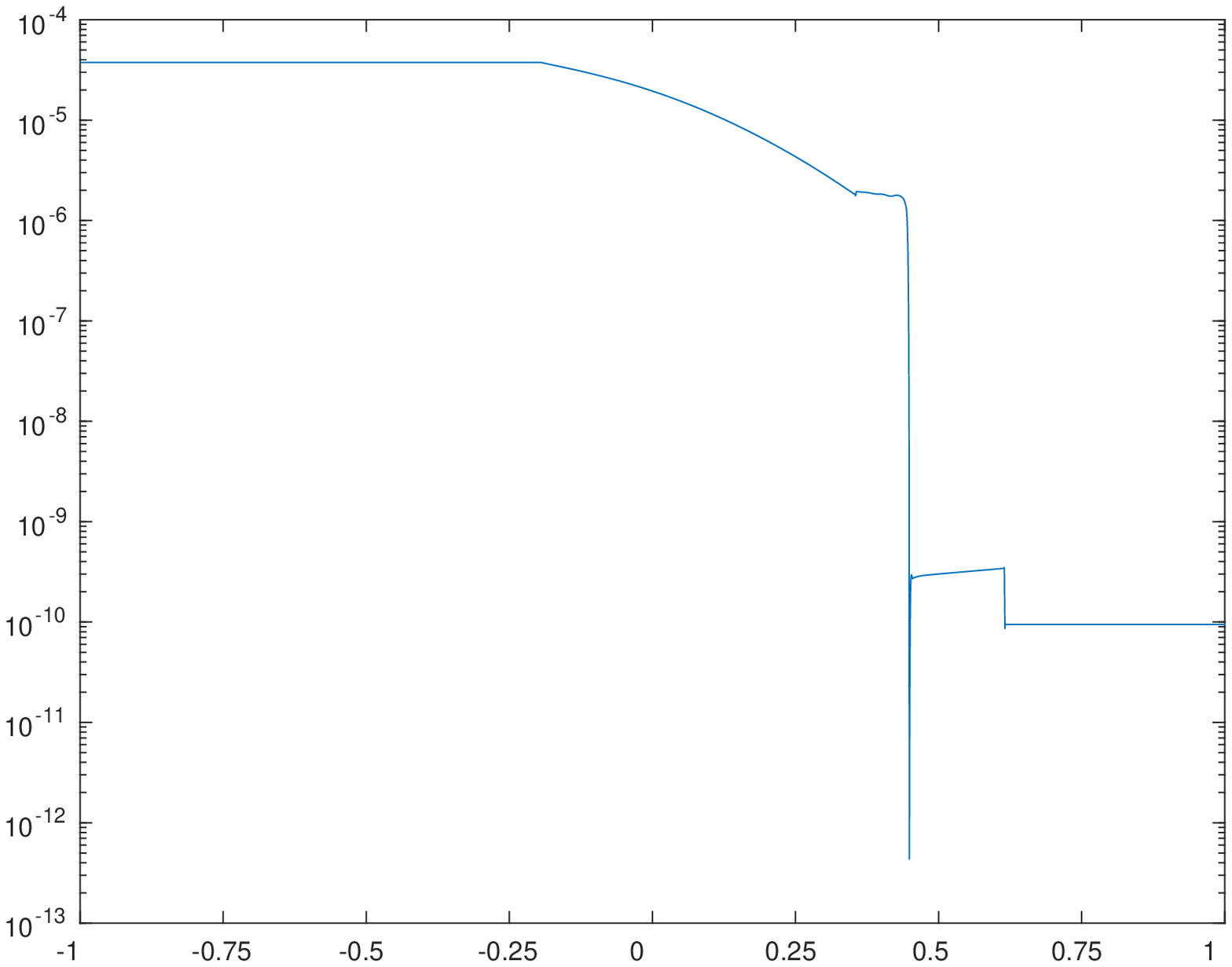}}\\
\subfigure[$ \rho_3$]{\includegraphics[width = 3in]{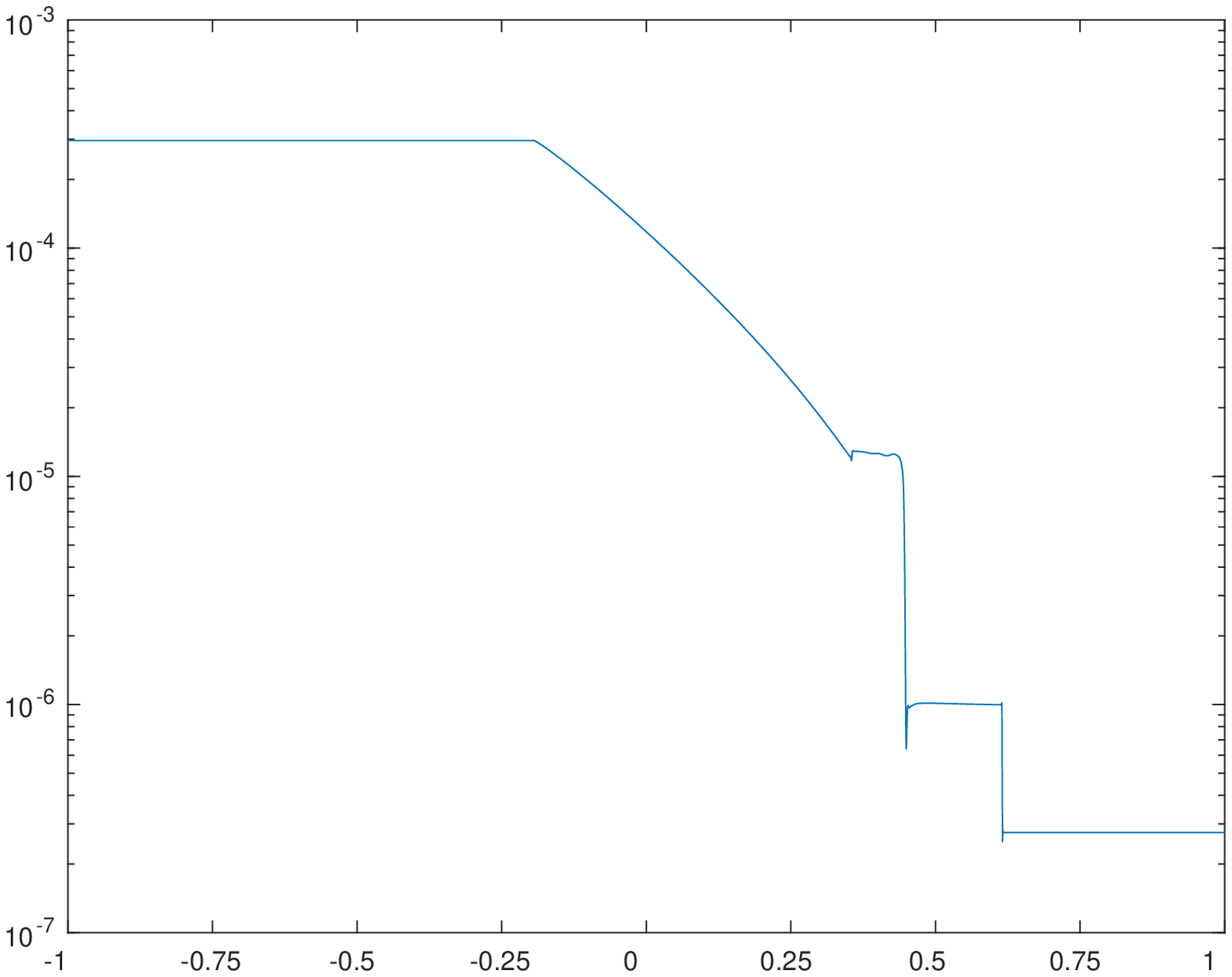}}
\subfigure[$ u$]{\includegraphics[width = 3in]{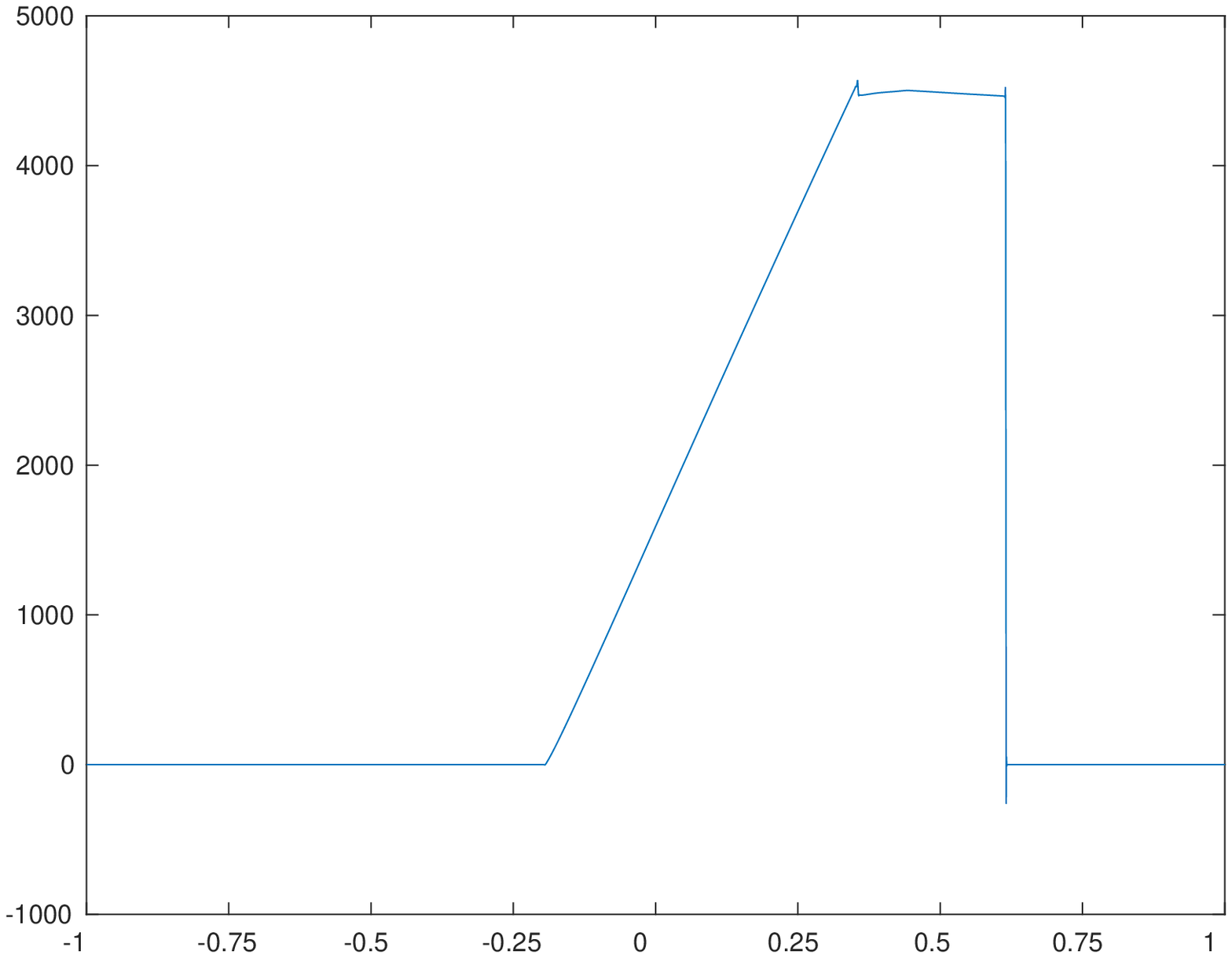}}\\
\subfigure[$ p$]{\includegraphics[width = 3in]{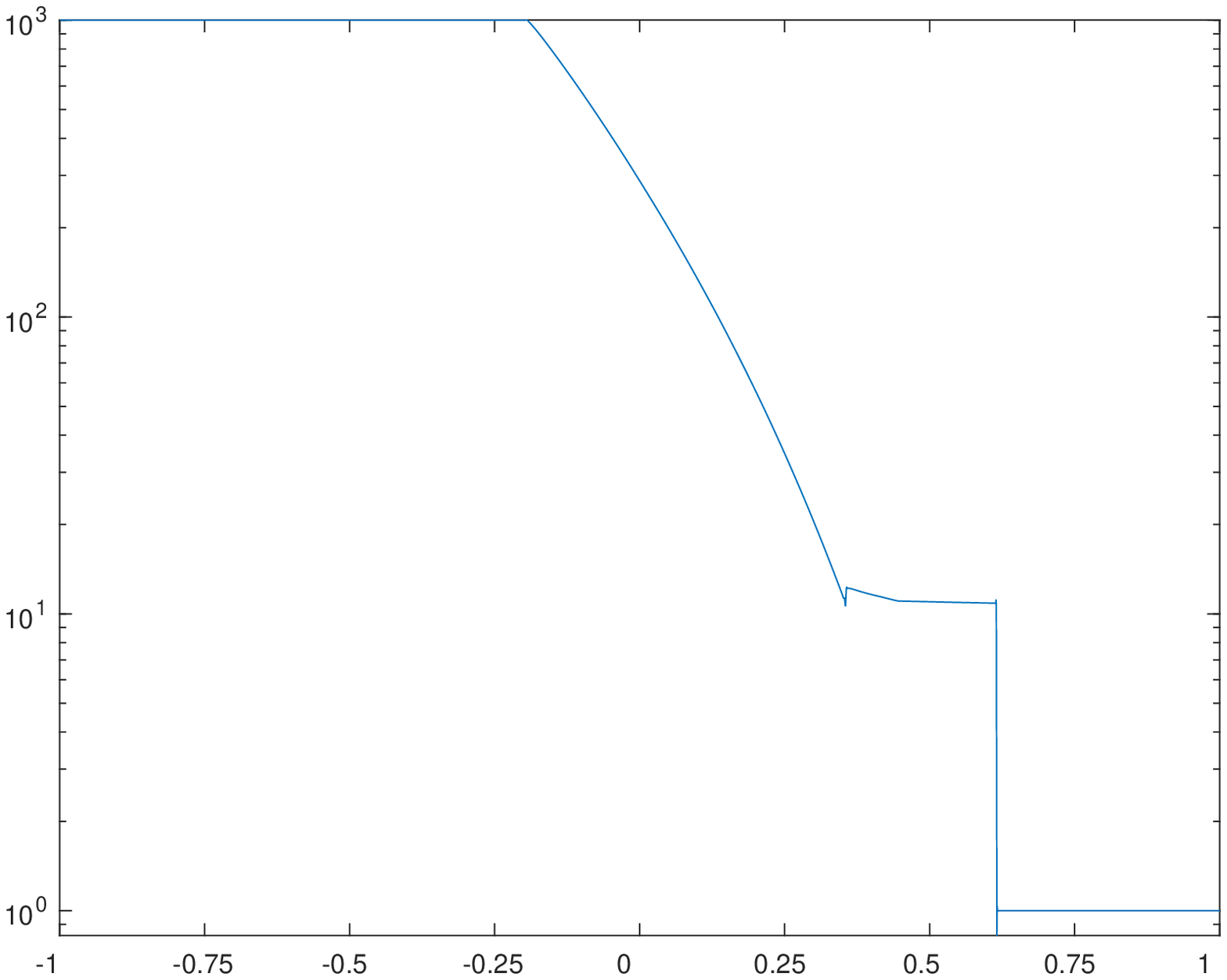}}
\caption{Example 4.3: three species reaction problem at t=0.0001 using MPMS2}
\label{4.3_MPMS2}
\end{figure}

\begin{figure}[!htbp]
\subfigure[$ \rho_1$]{\includegraphics[width = 3in]{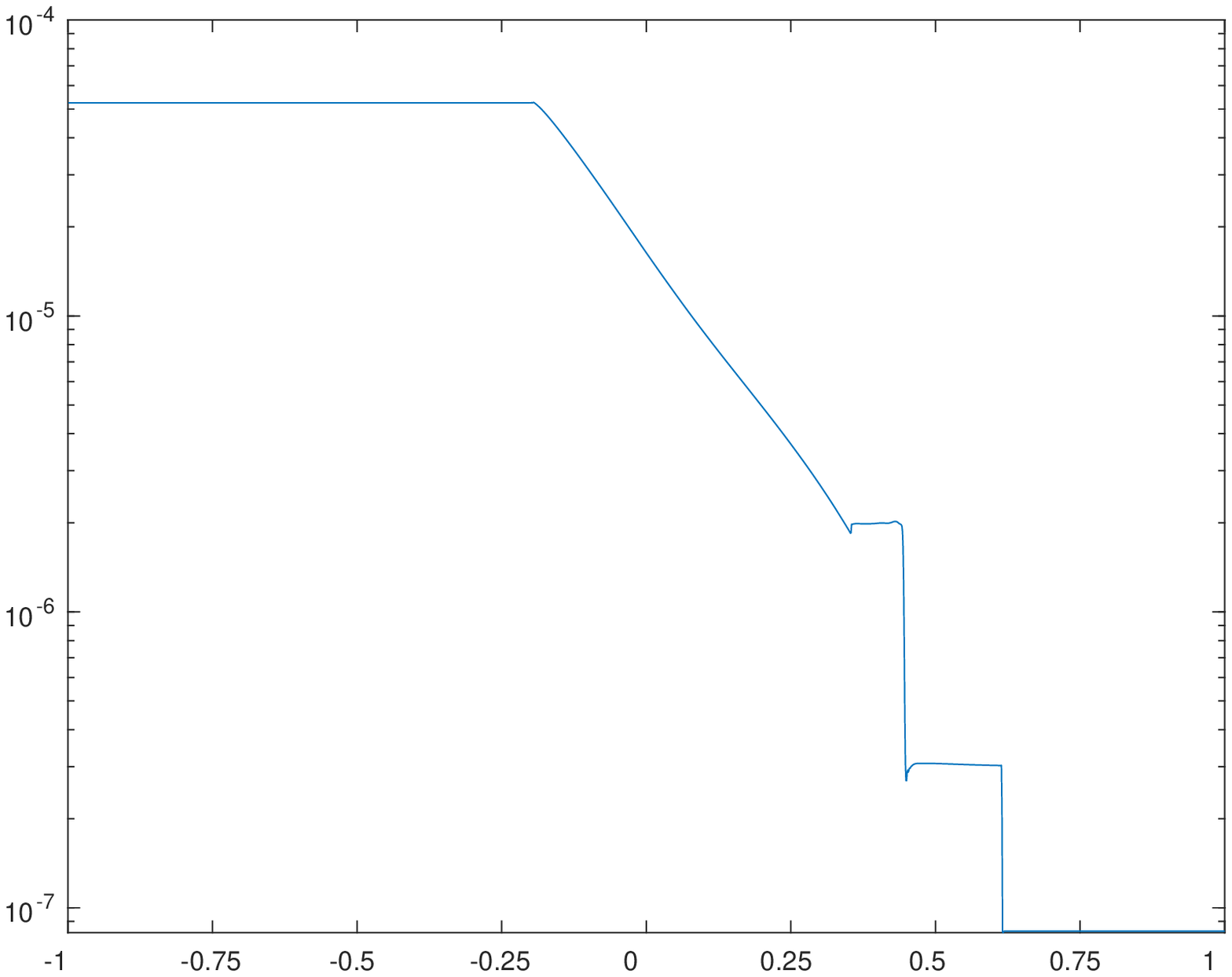}}
\subfigure[$ \rho_2$]{\includegraphics[width = 3in]{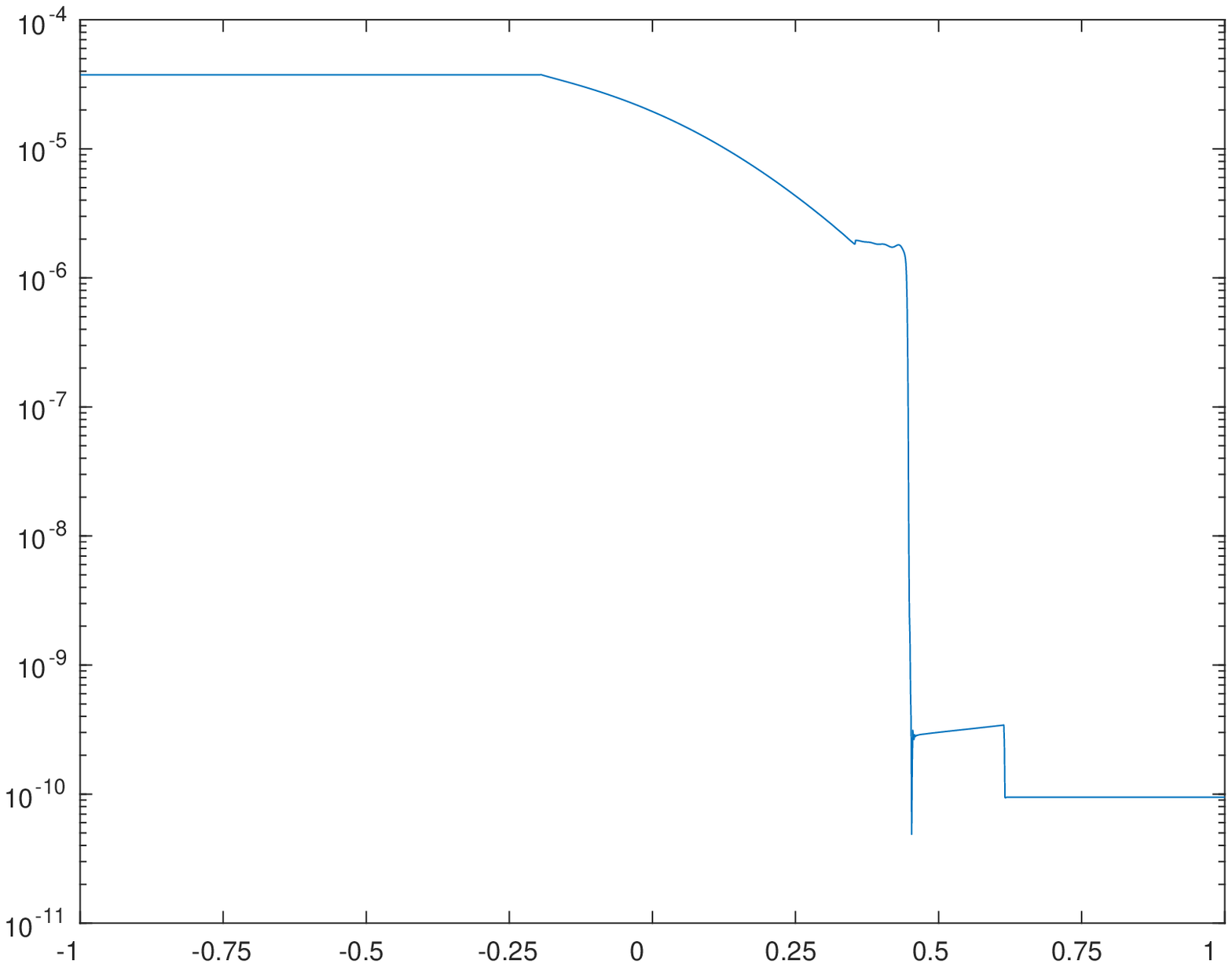}}\\
\subfigure[$ \rho_3$]{\includegraphics[width = 3in]{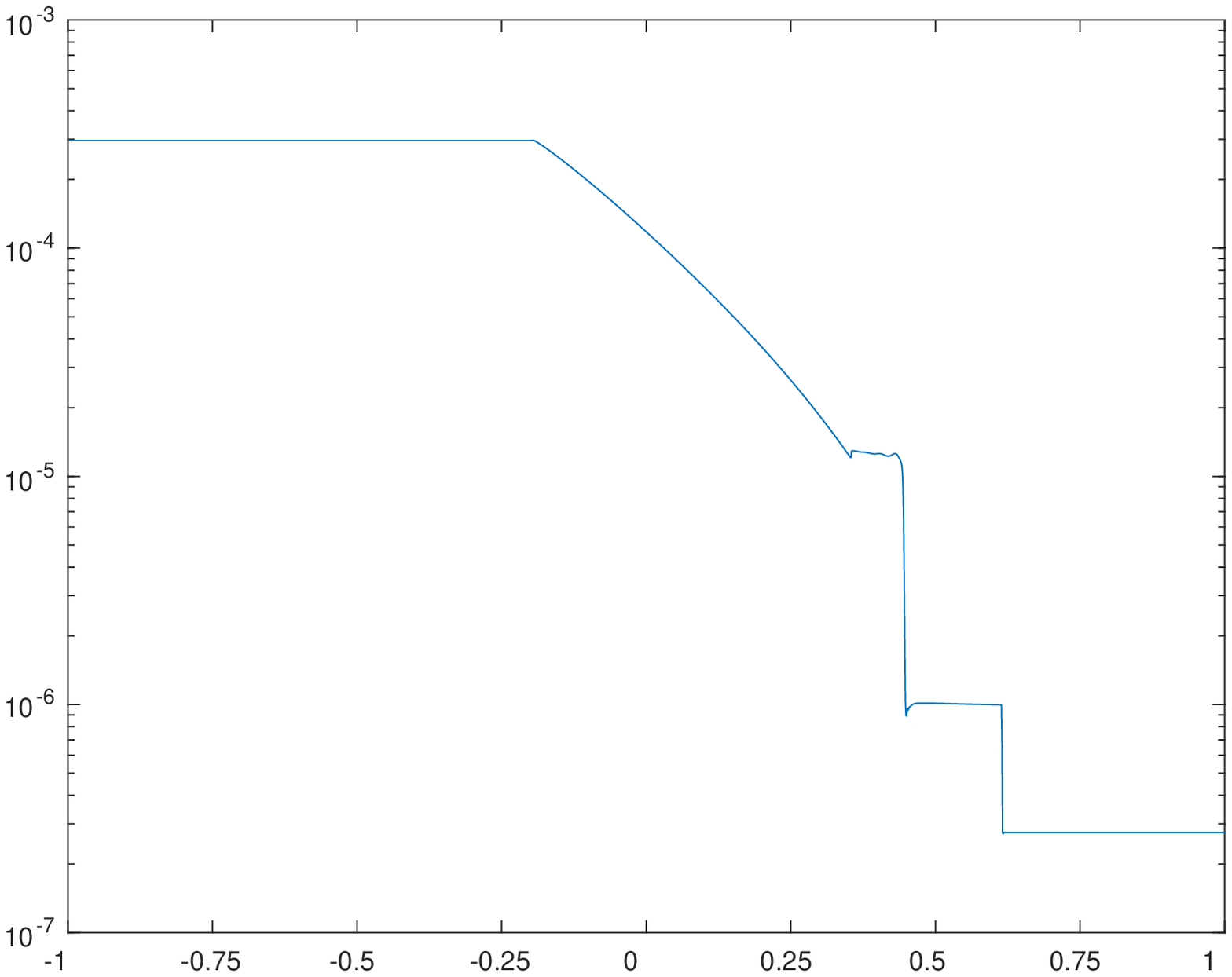}}
\subfigure[$ u$]{\includegraphics[width = 3in]{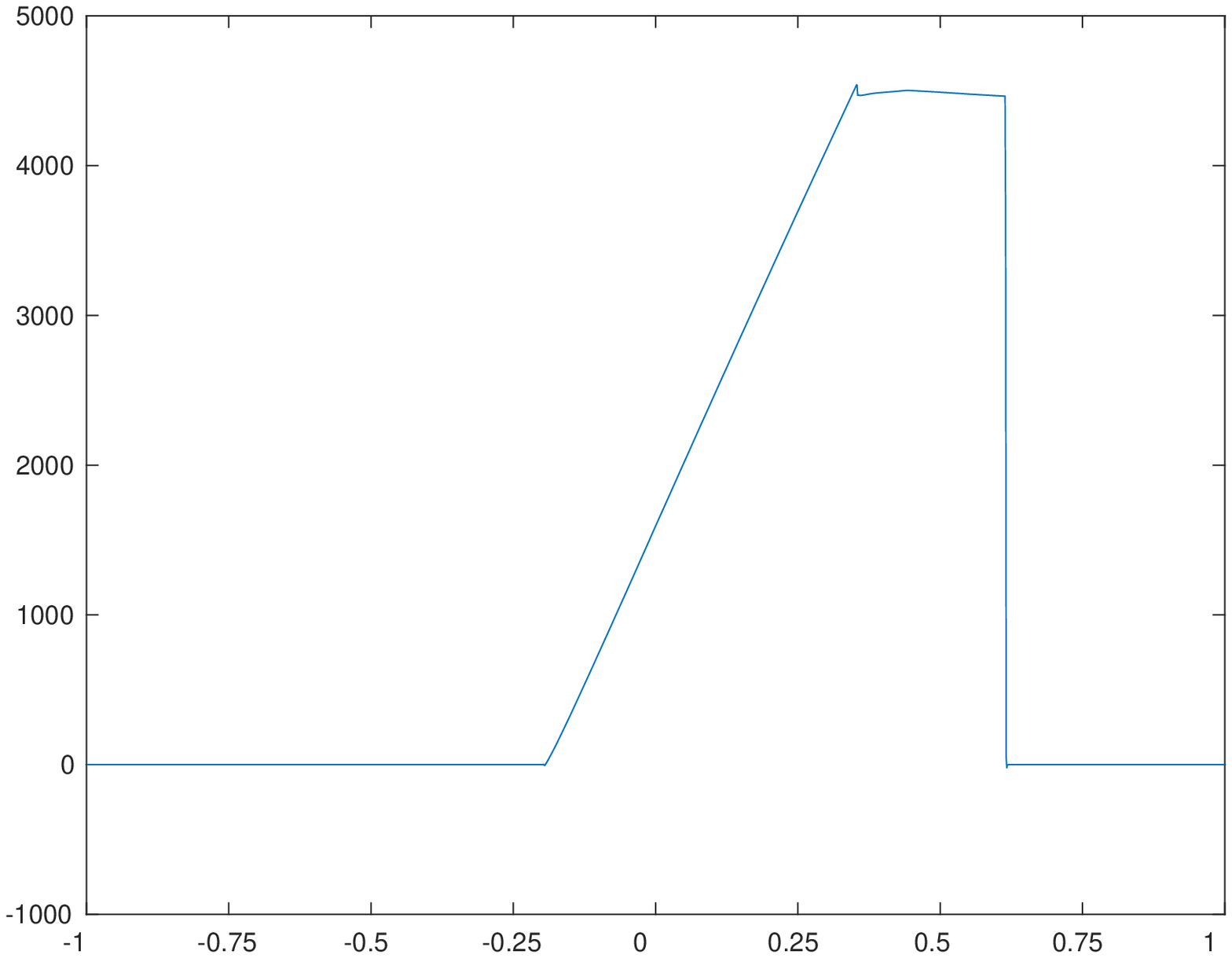}}\\
\subfigure[$ p$]{\includegraphics[width = 3in]{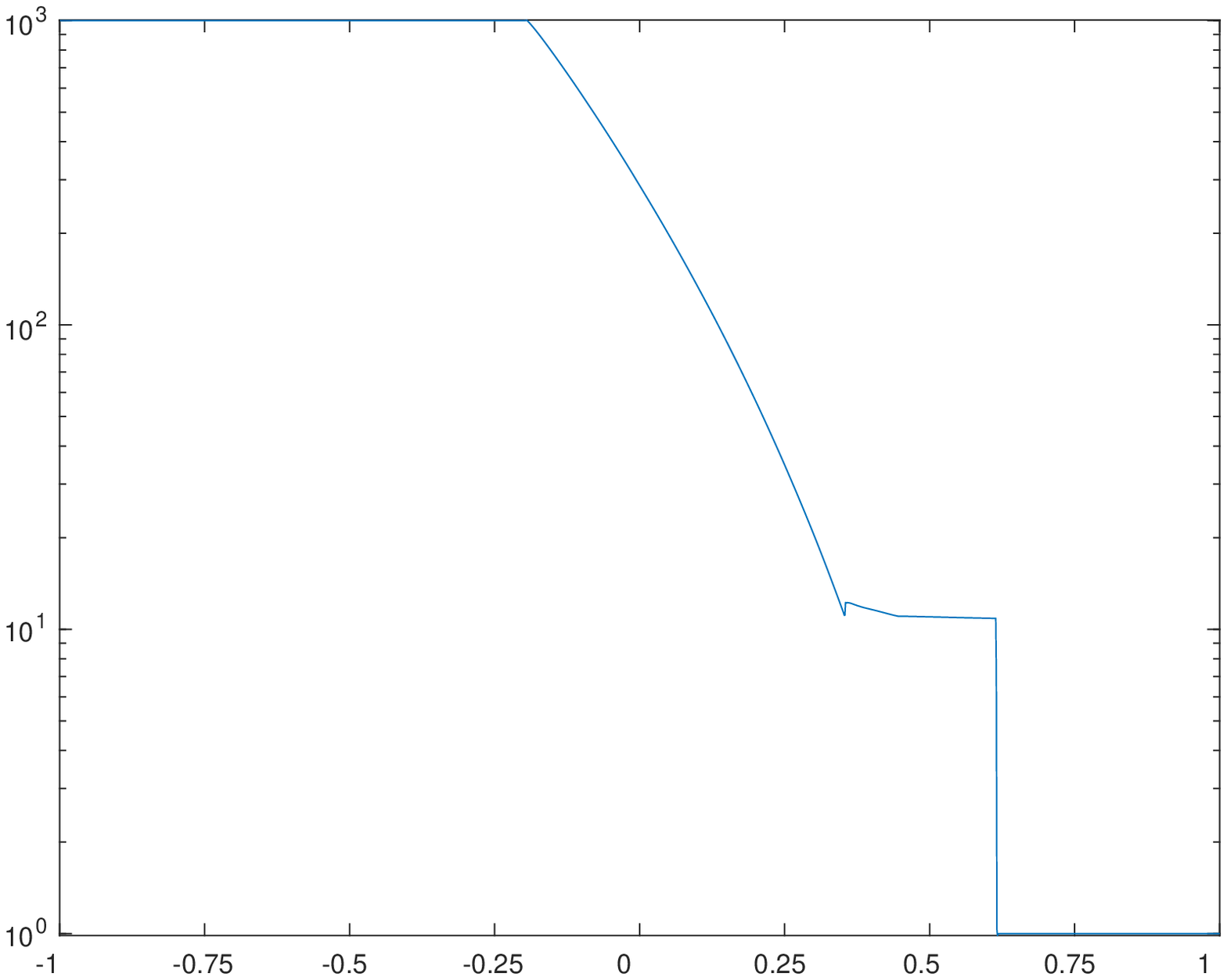}}
\caption{Example 4.3: three species reaction problem at t=0.0001 using MPMS3}
\label{4.3_MPMS3}
\end{figure}

\subsection{Reactive Euler Equations}
Consider the reactive Euler equations in 2D with
\begin{eqnarray}
\bm{U}&=&(\rho, m, n, E, \rho Y)^T, \notag\\
\bm{F}(\bm{U})&=&(m, \rho u^2+p, \rho uv, (E+p)u, \rho uY)^T, \notag\\
\bm{G}(\bm{U})&=&(n, \rho uv, \rho v^2+p, (E+p)v, \rho vY)^T, \notag\\
\bm{S}(\bm{U})&=&(0,0,0,0,\omega), \label{2D Euler equation}
\end{eqnarray}
where
$$
m=\rho u,\quad n=\rho v,\quad E=\frac12\rho (u^2+v^2)+\frac{p}{\gamma-1}+\rho qY,
$$
with $Y$ denoting the reactant mass fraction and $\omega$ being the source term which is expressed as
$$
\omega=-\tilde{K}\rho Ye^{-\tilde{T}/T}.
$$
To fit into our framework, we rewrite \eqref{2D Euler equation} into
\begin{eqnarray*}
\bm{U}&=&(\rho Y,\rho Z, m, n, E)^T, \notag\\
\bm{F}(\bm{U})&=&(\rho uY, \rho uZ, \rho u^2+p, \rho uv, (E+p)u)^T, \notag\\
\bm{G}(\bm{U})&=&(\rho vY, \rho vZ, \rho uv, \rho v^2+p, (E+p)v)^T, \notag\\
\bm{S}(\bm{U})&=&(\omega, -\omega, 0,0,0), 
\end{eqnarray*}
where $Z$ denotes the unreacted mass fraction. 
\begin{ex}\label{ex4}
In this example, the numerical convergence of our scheme is tested. The domain is set to be $[0, 2] \times [0, 2]$.
The initial condition is given as follows: if $x^2 + y^2 \leq 0.36$, then $(\rho, u, v, p, Y ) = (1, 0, 0, 80, 0)$; otherwise,
$(\rho, u, v, p, Y ) = (1, 0, 0, 10^{-9}, 1)$. The boundary conditions on the bottom and left are reflective. The terminal time is t = 0.2. We use uniform rectangular meshes with mesh sizes $\Delta x=\Delta y=1/120$. The numerical results with MPMS2 and MPRK2 are shown in Figure \ref{4.4_MPMS2} and Figure \ref{4.4_MPRK2}, respectively, which are comparable to the results in \cite{PP/detonation} where the converged solutions are observed. Also, we can observe some spurious oscillations in the numerical approximations, and this is mainly due to the lack of mechanisms to suppress oscillations.

\end{ex}

\begin{ex}\label{ex5}
We test the detonation diffraction in this example. The same parameters and initial conditions in \cite{PP/detonation} are applied. 
The initial condition are as follows: if $x<0.5$, then $(\rho,u,v,E,Y)=(11,6.18, 0, 970, 1)$; otherwise, $(\rho,u,v,E,Y)=(1, 0, 0, 5, 5, 1)$. The boundary conditions are reflective except that at $x=0$, where we take $(\rho,u,v,E,Y)=(11,6.18, 0, 970, 1)$. The terminal time is $t=0.6$. The parameters are $\gamma=1.2$, $q=50$, $\tilde{T}=50$ and $\tilde{K}=2566.4$. The numerical schemes may produce negative density and or pressure for this example which can lead to blow-up of the numerical simulations. DG coupled with the time integration MPRK2, and MPMS2 are tested and the numerical results with $\Delta x=\Delta y=1/48$ are shown in Figures \ref{4.5_MPRK2}-\ref{4.5_MPMS2}. Our numerical results agree with previous works.
\end{ex}

\begin{figure}[!htbp]
\centering
\subfigure[Contour of Density]{\includegraphics[width = 3.2in]{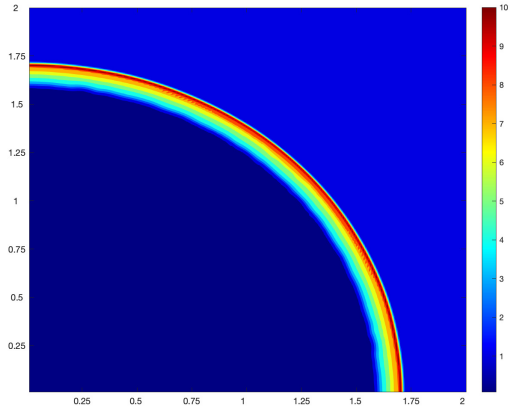}}
\subfigure[Contour of Pressure]{\includegraphics[width = 3.2in]{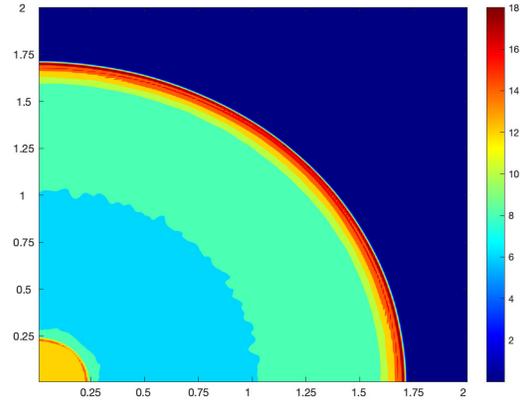}}\\
\subfigure[Cut along $y=0$]{\includegraphics[width = 3in]{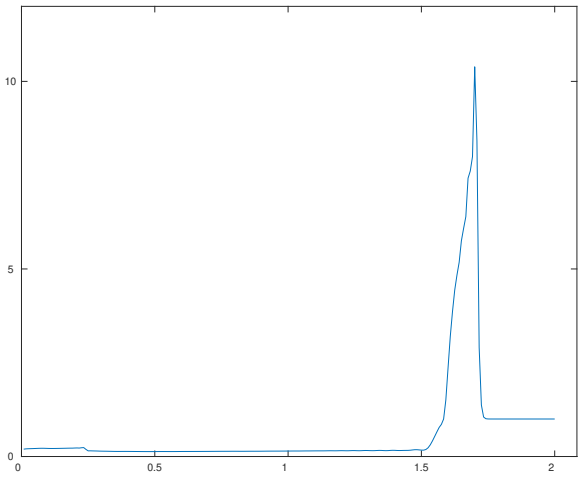}}
\subfigure[Cut along $y=0$]{\includegraphics[width = 3.08in]{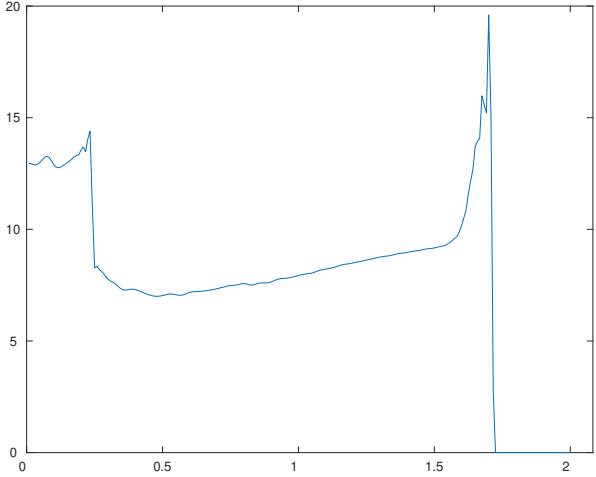}}
\caption{Example4.4 numerical convergence study MPMS2}
\label{4.4_MPMS2}
\end{figure}

\begin{figure}[!htbp]
\centering
\subfigure[Contour of Density]{\includegraphics[width = 3.2in]{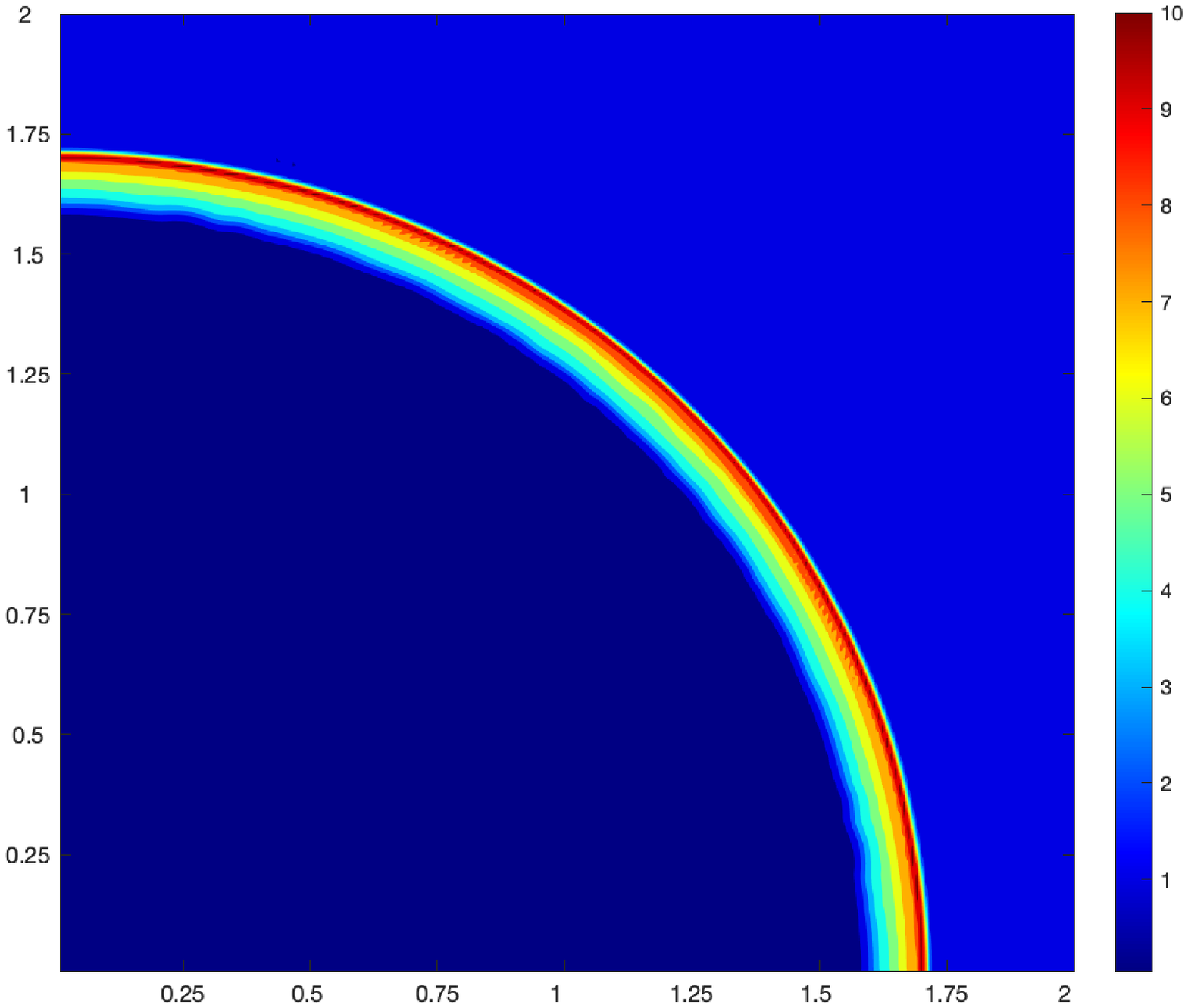}}
\subfigure[Contour of Pressure]{\includegraphics[width = 3.2in,height=2.6in]{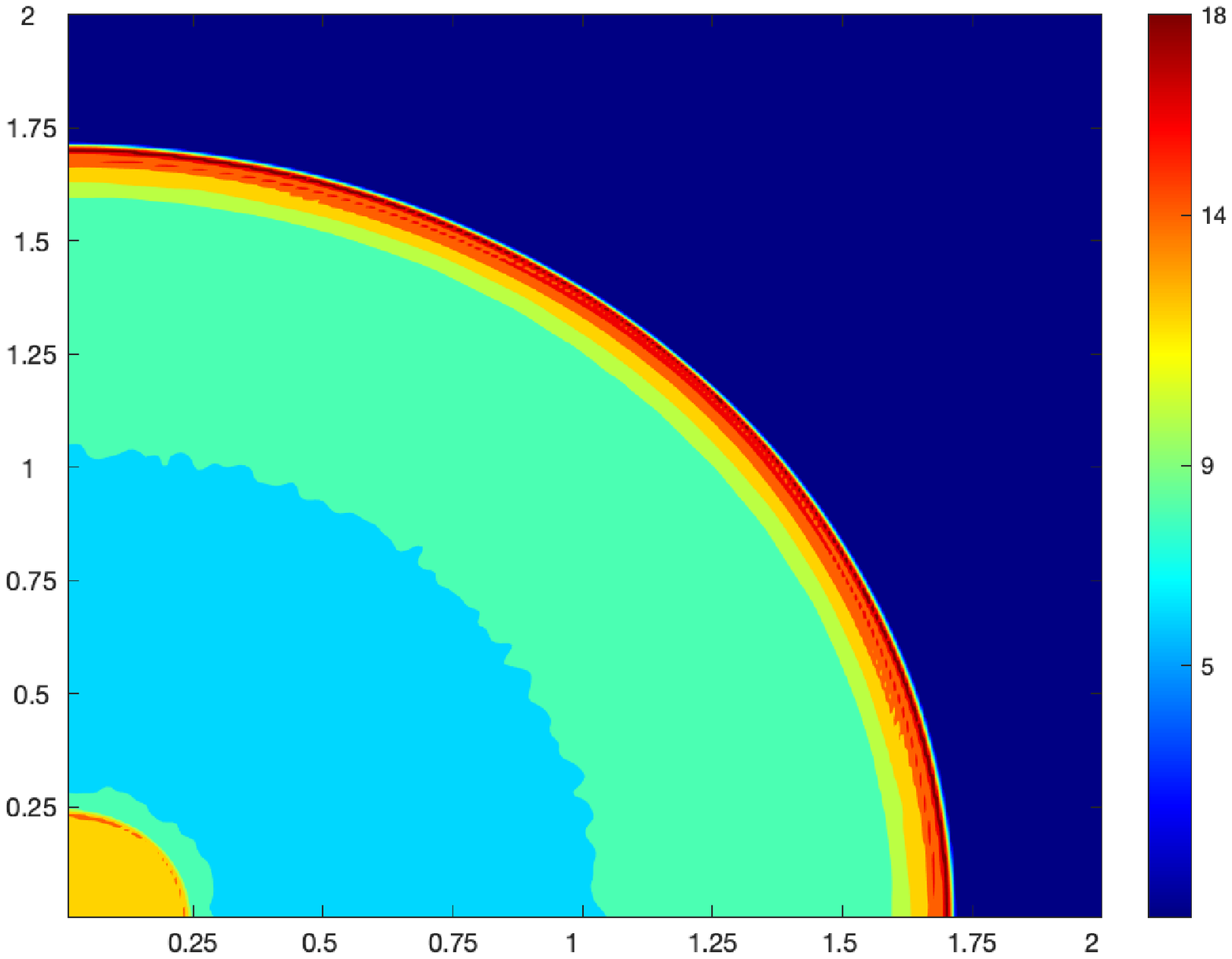}}\\
\subfigure[Cut along $y=0$]{\includegraphics[width = 3in]{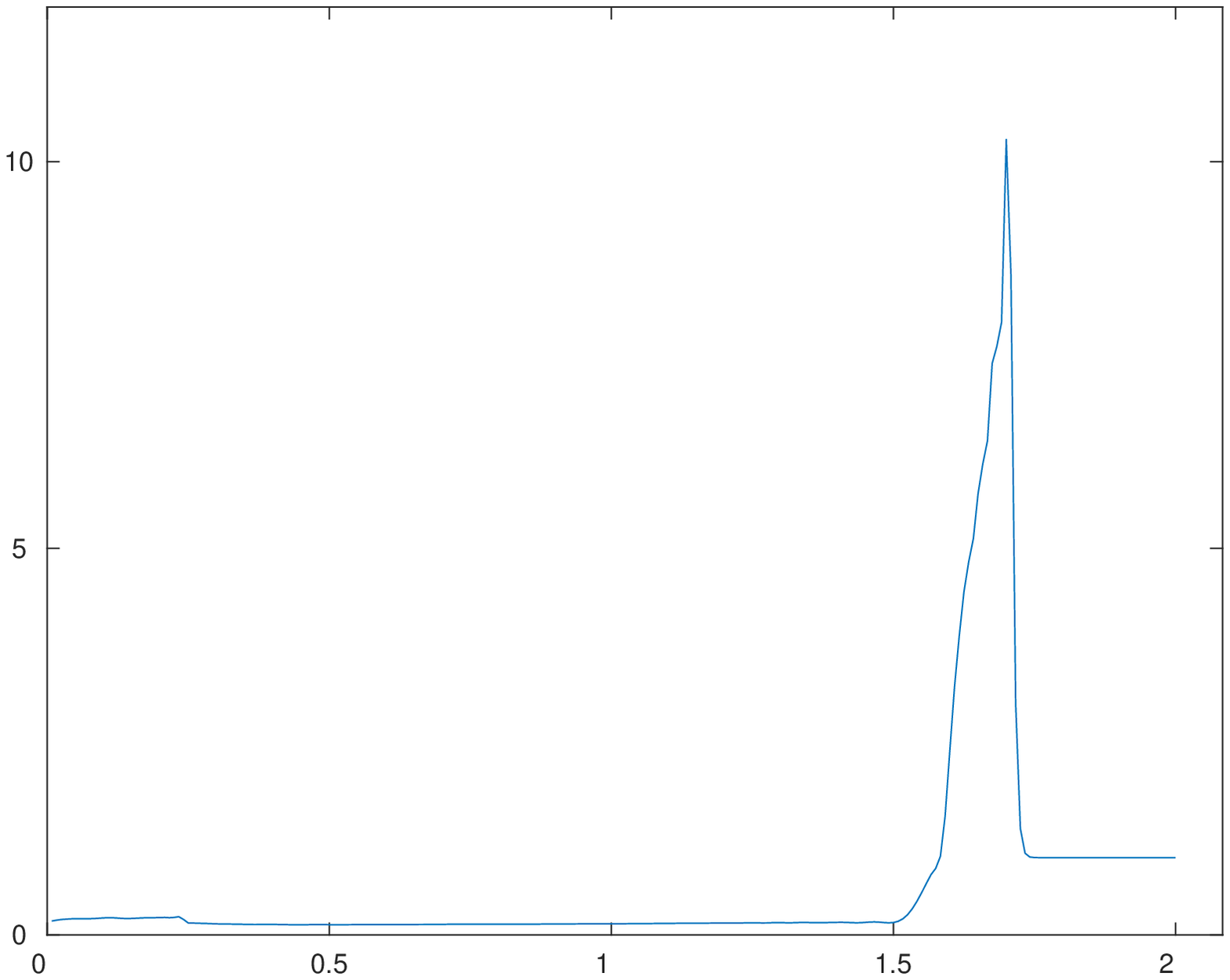}}
\subfigure[Cut along $y=0$]{\includegraphics[width = 3in]{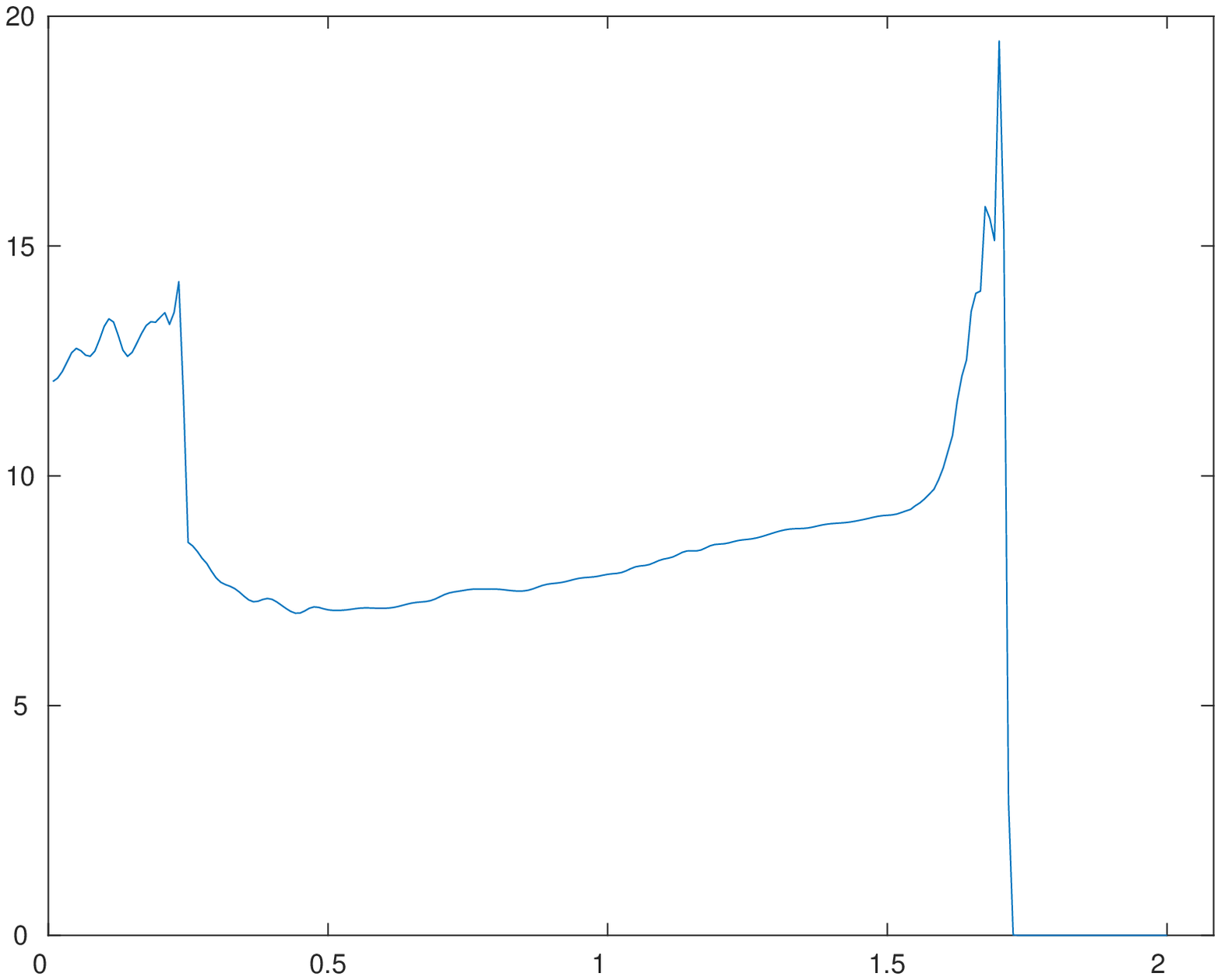}}
\caption{Example4.4 numerical convergence study MPRK2}
\label{4.4_MPRK2}
\end{figure}

\begin{figure}[!htbp]
\subfigure[Colored Contour of Density]{\includegraphics[width = 3.75in]{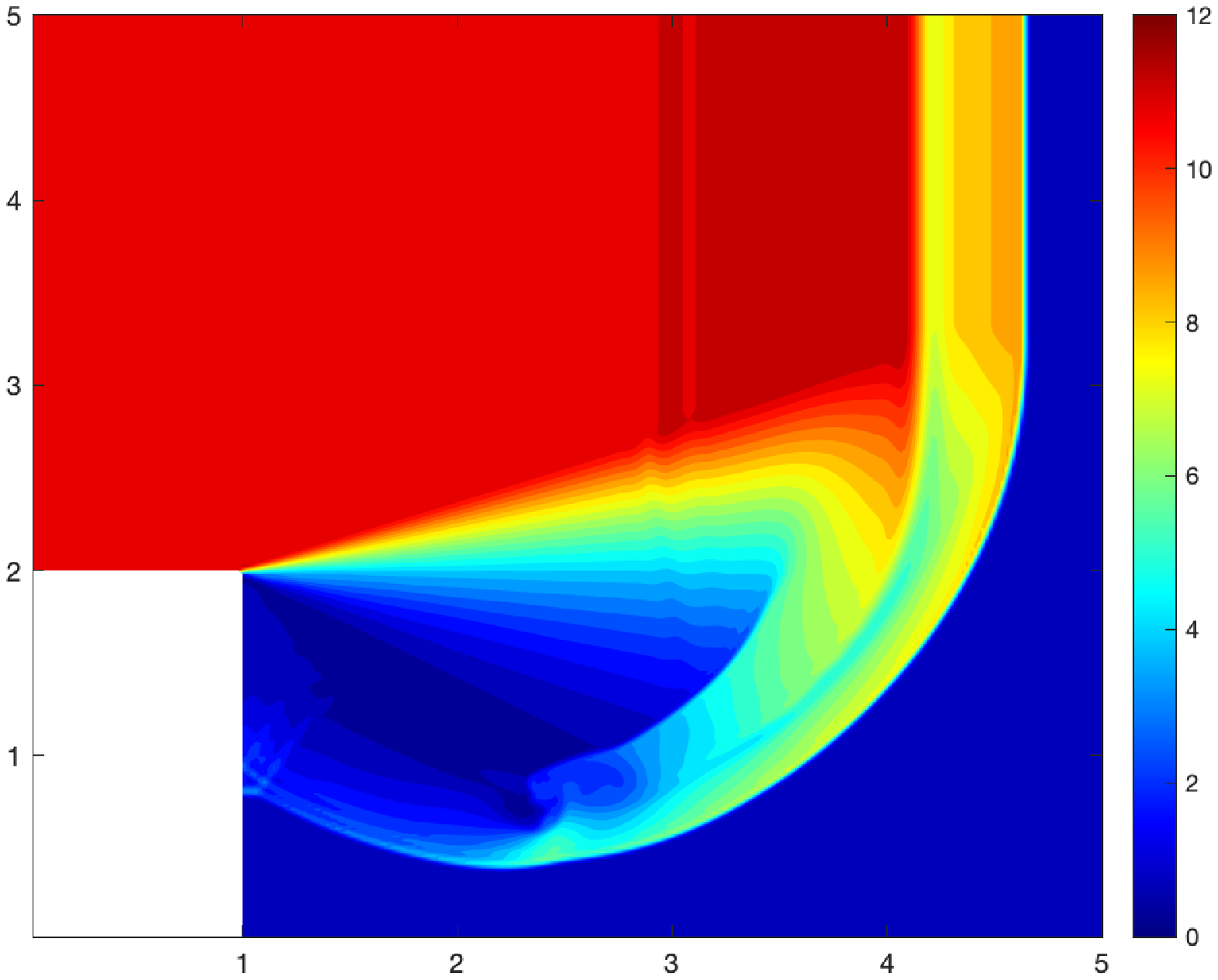}}
\subfigure[Contour Line of Density ]{\includegraphics[width = 3in]{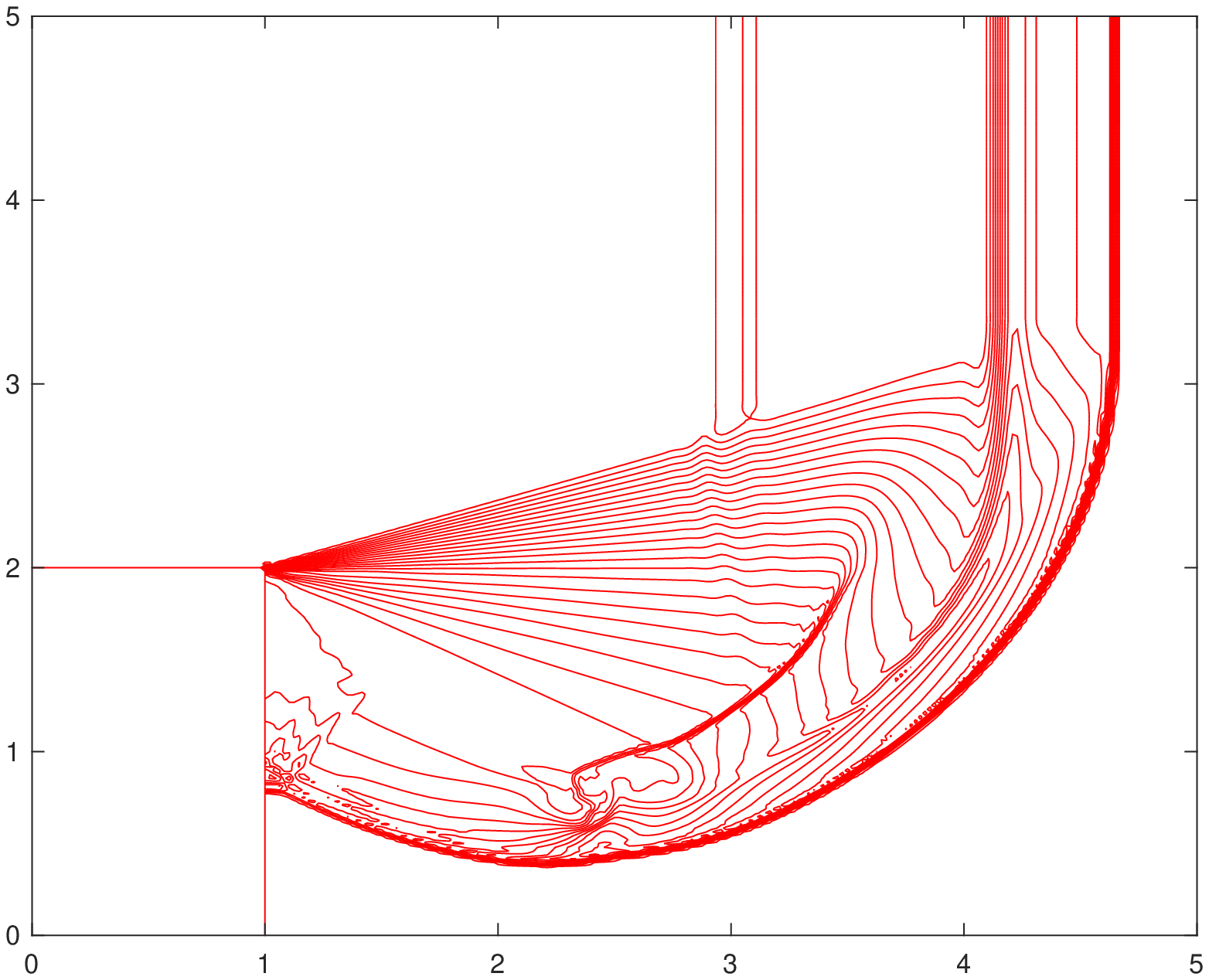}}\\
\subfigure[Colored Contour of Pressure]{\includegraphics[width = 3.75in]{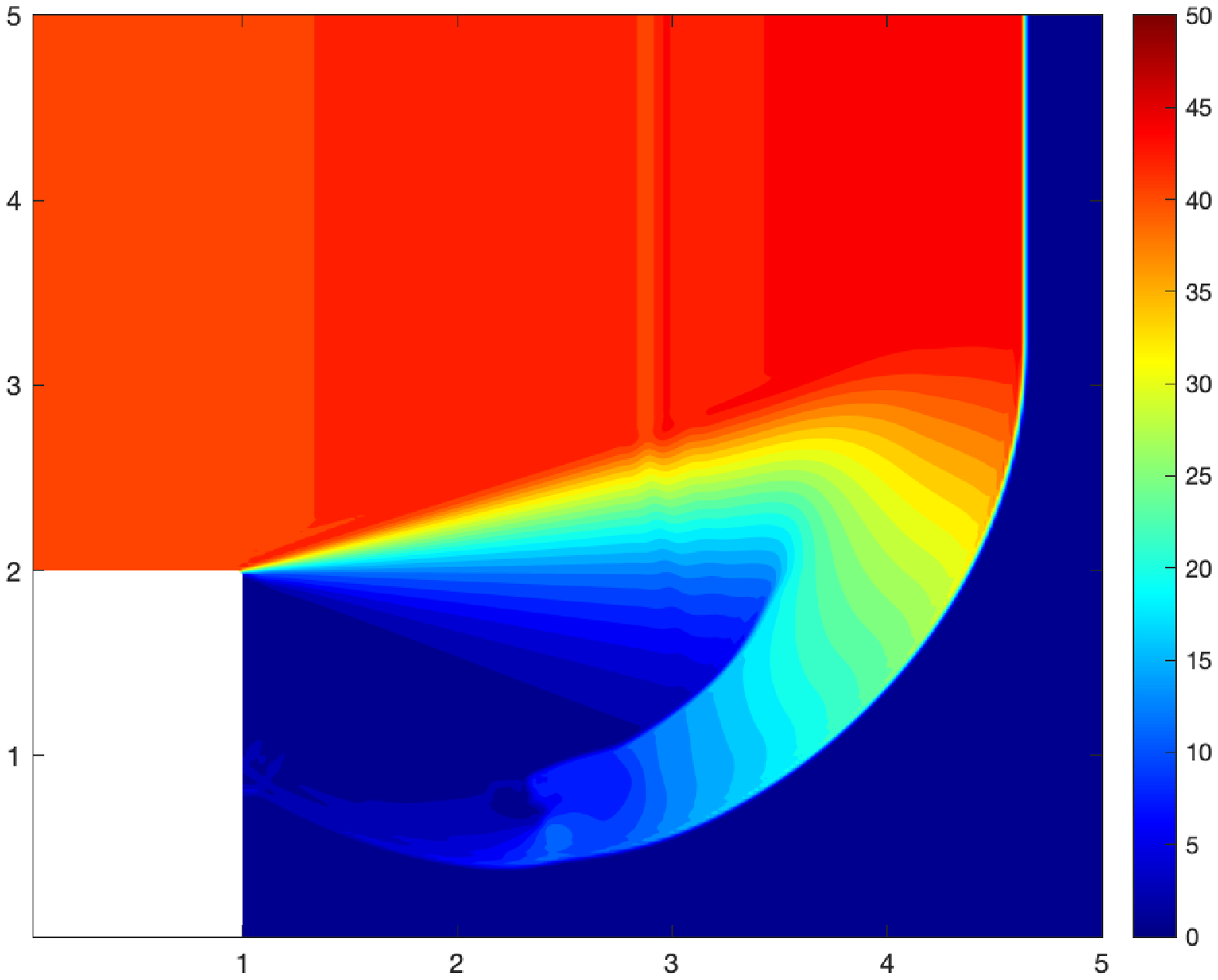}}
\subfigure[Contour Line of Pressure]{\includegraphics[width = 3in]{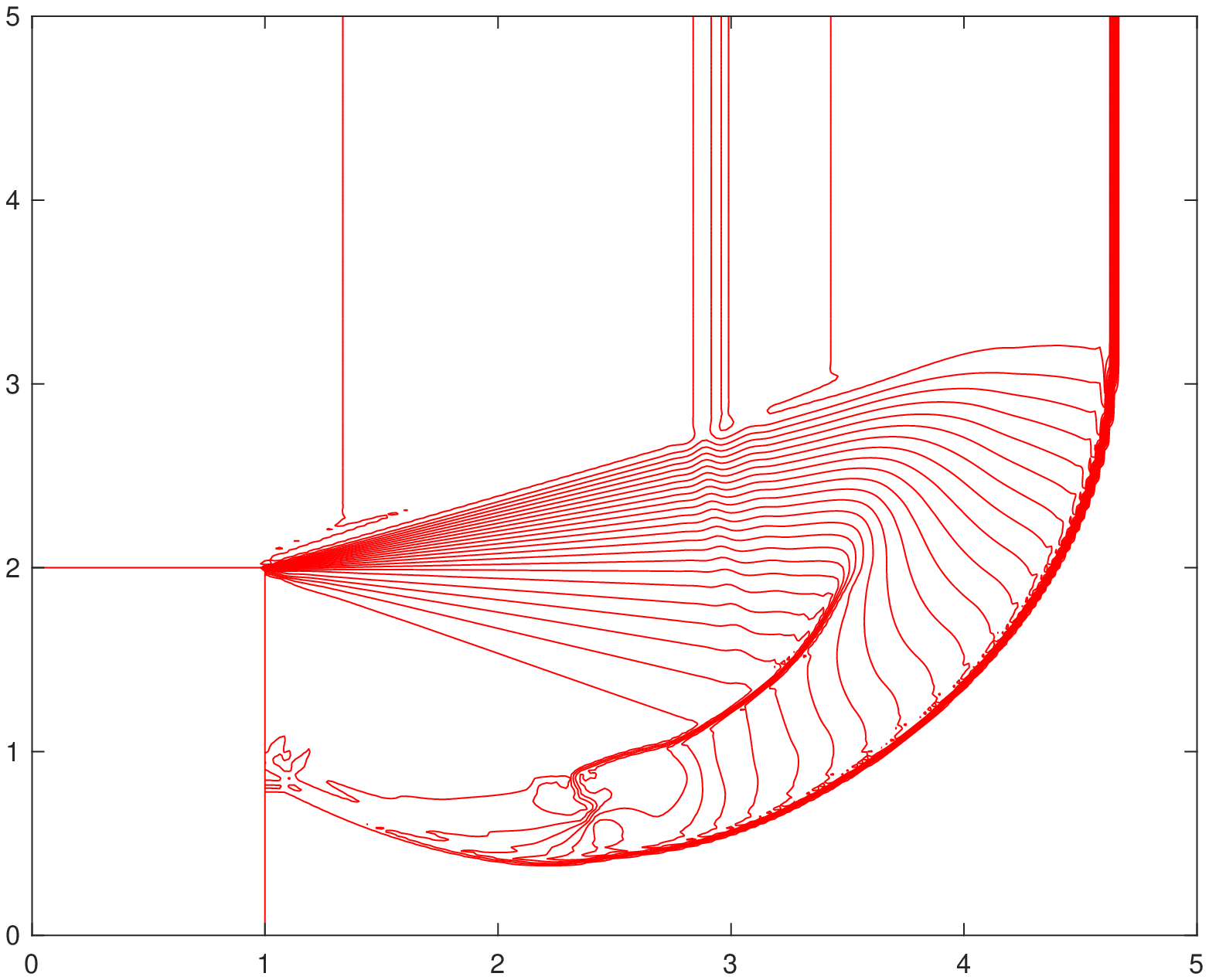}}
\caption{Example4.5 detonation diffraction problem using MPRK2}
\label{4.5_MPRK2}
\end{figure}

\begin{figure}[!htbp]
\subfigure[Colored Contour of Density]{\includegraphics[width = 3.75in]{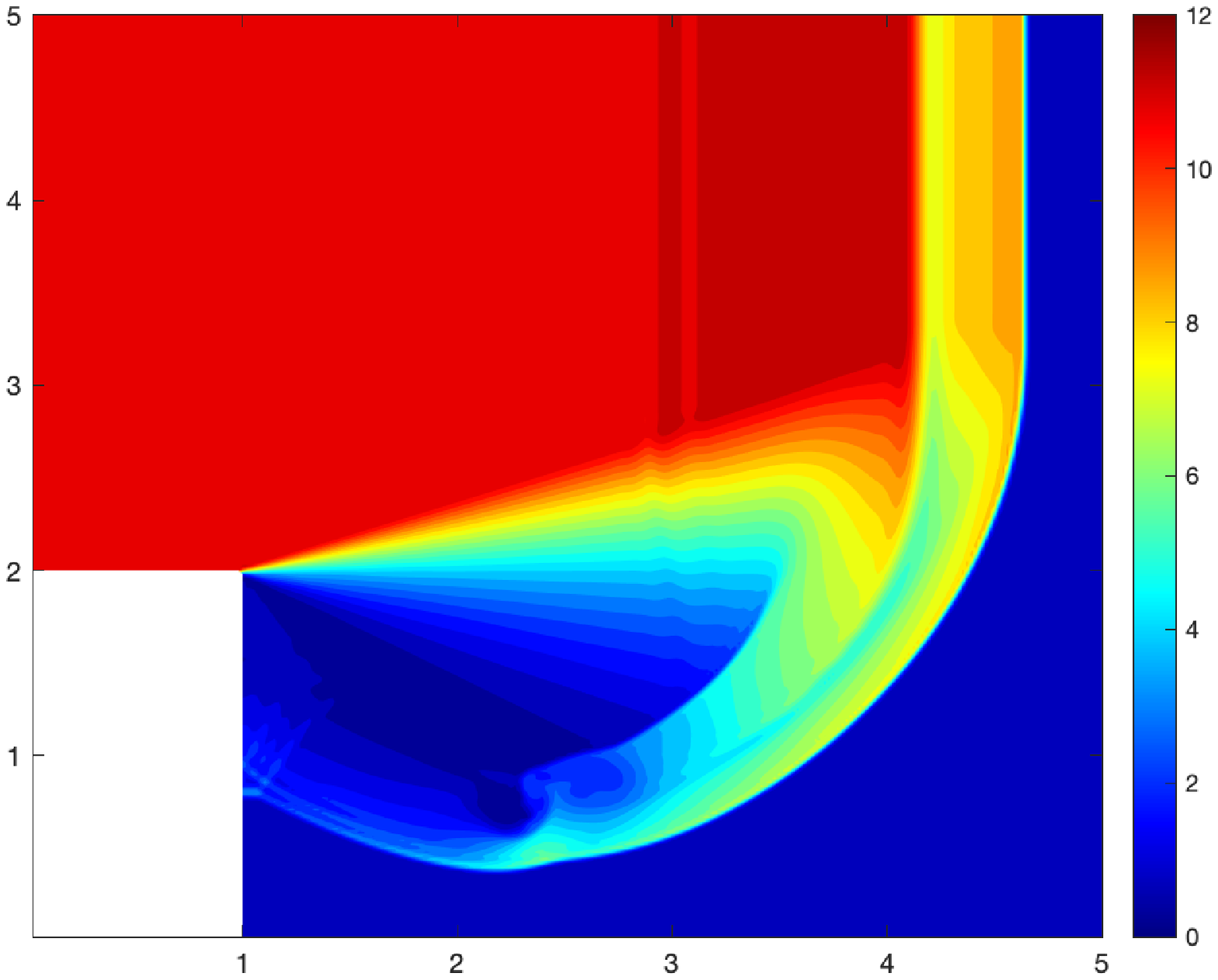}}
\subfigure[Contour Line of Density ]{\includegraphics[width = 3in]{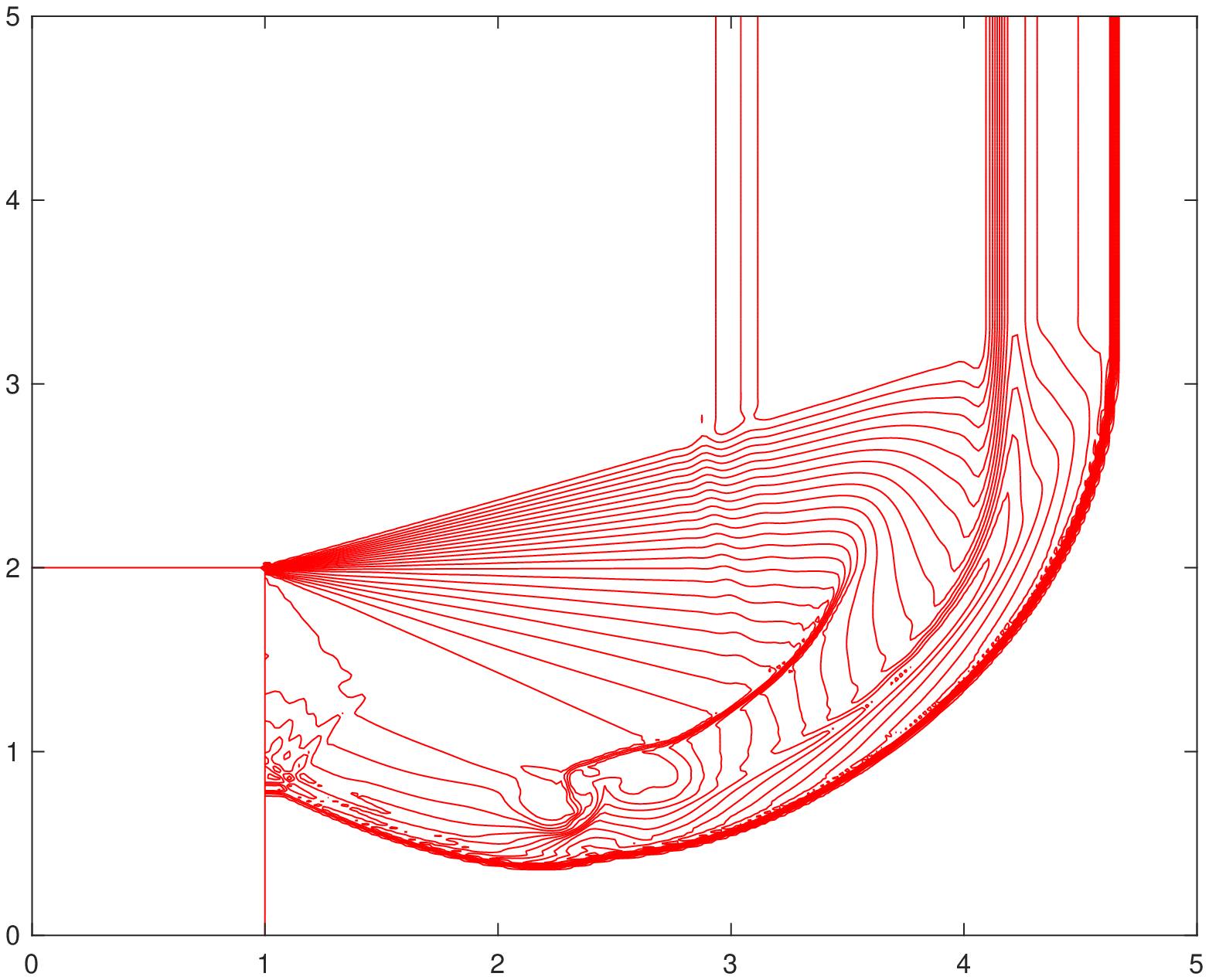}}\\
\subfigure[Colored Contour of Pressure]{\includegraphics[width = 3.75in]{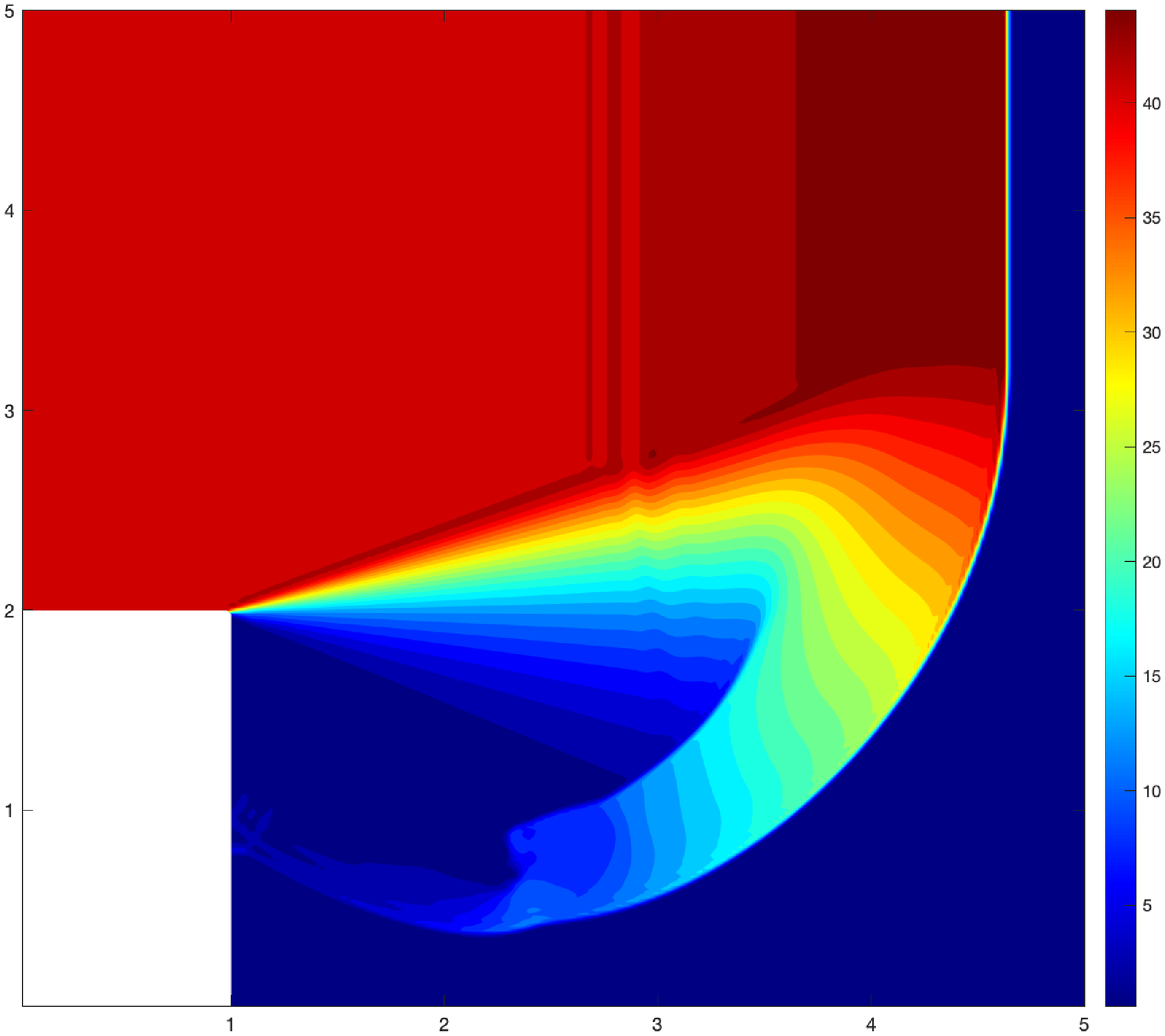}}
\subfigure[Contour Line of Pressure]{\includegraphics[width = 3in]{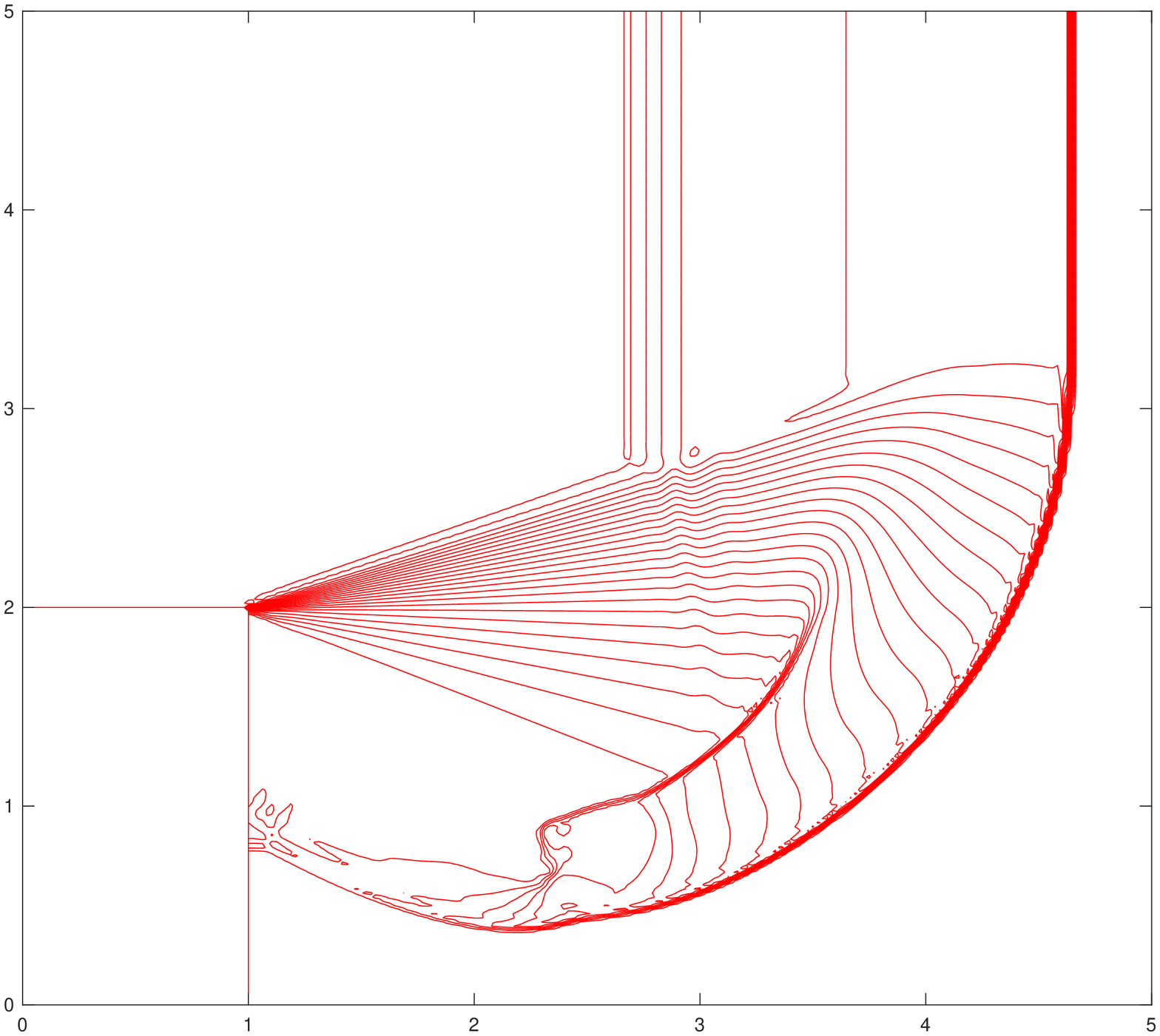}}
\caption{Example 4.5 detonation diffraction problem using MPMS2}
\label{4.5_MPMS2}
\end{figure}

\section{Concluding remarks}
In this paper, we constructed MP MS method up to third order accuracy for production-destruction equations. Coupled with bound-preserving DG methods, the scheme we obtained have both conservative and bound-preserving properties. Numerical experiments have shown the accuracy and effectiveness of the proposed schemes. {
In the numerical experiments, we take appropriately small time step. To the best of our knowledge, there is no modified Patankar multistep method in the literature. Therefore, our method is a supplement to the existing method. The multistep methods generally require less function evaluations in each step than the Runge-Kutta method. Therefore, for ODEs or PDEs without characteristic speed changing wildly, the multistep method works better than the Runge-Kutta method. However, if the velocity changes wildly, the Runge-Kutta method would perform better.}

\section*{Conflict of interest}
On behalf of all authors, the corresponding author states that there is no conflict of interest.


\begin{thebibliography}{}
\bibitem{modified Patankar} H. Burchard, E. Deleersnijder, A. Meister, {\em A high-order conservative Patankar-type discretisation for stiff systems of production–destruction equations}. Applied Numerical Mathematics, v47 (2003), pp. 1–30 

\bibitem{Chuenjarern} N. Chuenjarern, Z. Xu and Y. Yang, {\em High-order bound-preserving discontinuous Galerkin methods for compressible miscible displacements in porous media on triangular meshes}, Journal of Computational Physics, v378 (2019), pp.110-128.

\bibitem{Cockburn I} B. Cockburn and C.-W. Shu, {\em The Runge-Kutta local projection P1-discontinuous- Galerkin finite element method for scalar conservation laws}, Mathematical Modelling and Numerical Analysis (M2AN), v25 (1991), pp. 337-361.

\bibitem{Cockburn II} B. Cockburn and C.-W. Shu, {\em TVB Runge-Kutta local projection discontinuous Galerkin finite element method for conservation laws II: general framework}, Mathematics of Computation, v52 (1989), pp. 411-435.

\bibitem{Cockburn III} B. Cockburn, S.-Y. Lin and C.-W. Shu, {\em TVB Runge-Kutta local projection discontinuous Galerkin finite element method for conservation laws III: one dimensional systems}, Journal of Computational Physics, v84 (1989), pp. 90-113.

\bibitem{Cockburn IV} B. Cockburn, S. Hou and C.-W. Shu, {\em The Runge-Kutta local projection discontinuous Galerkin finite element method for conservation laws IV: the multidimensional case}, Mathematics of Computation, v54 (1990), pp. 545-581.

\bibitem{Cockburn V} B. Cockburn and C.-W. Shu, {\em The Runge-Kutta discontinuous Galerkin method for conservation laws V: multidimensional systems}, Journal of Computational Physics, v141 (1998), pp. 199-224.

\bibitem{DuCAMC} J. Du and Y. Yang, {\em High-order bound-preserving finite difference methods for multispecies and multireaction detonations}, Communications on Applied Mathematics and Computation, accepted.

\bibitem{JDdetonation1} J. Du, C. Wang, C. Qian, Y. Yang, \emph{High-order bound-preserving discontinuous Galerkin methods for stiff multispecies detonation}, SIAM Journal on Scientific Computing,
v41 (2019), pp. B250–B273.

\bibitem{JDdetonation2} J. Du and Y. Yang, {\em Third-order conservative sign-preserving and steady-state-preserving time integrations and applications in stiff multispecies and multireaction detonations}, Journal of Computational Physics, v395 (2019), pp. 489-510.

\bibitem{SSP} S. Gottlieb, C.-W. Shu, and E. Tadmor. \emph{Strong stability-preserving high-order time discretization methods}, SIAM review, 43 (2001), pp. 89–112.

\bibitem{bound-preserving} H. Guo and Y. Yang, {\em Bound-preserving discontinuous Galerkin method for compressible miscible displacement problem in porous media}, SIAM Journal on Scientific Computing, v39 (2017), pp. A1969– A1990.

\bibitem{Izgin3} J. Huang, T. Izgin, S. Kopecz, A. Meister and C.-W. Shu, {\em On the stability of strong-stability-preserving modified Patankar Runge-Kutta schemes}. arXiv preprint arXiv:2205.01488 (2022).

\bibitem{MPRK2} J. Huang, C.-W. Shu,  \emph{Positivity-preserving time discretizations for production-destruction equations with applications to non-equilibrium flows}. Journal of Scientific Computing, v78 (2019), pp. 1181-1839.

\bibitem{MPRK3} J. Huang, W. Zhao and C.-W. Shu, {\em A third-order unconditionally positivity-preserving scheme for production-destruction equations with applications to non-equilibrium flows} , Journal of Scientific Computing, v79 (2019), pp. 1015-1056.

\bibitem{Izgin1} T. Izgin, S. Kopecz and A. Meister, {\em On Lyapunov stability of positive and conservative time integrators and application to second order modified Patankar–Runge–Kutta schemes}, ESAIM: Mathematical Modelling and Numerical Analysis, v56(2022), 1053-1080.

\bibitem{Izgin2} T. Izgin, S. Kopecz and A. Meister, {\em On the Stability of Unconditionally Positive and Linear Invariants Preserving Time Integration Schemes}. arXiv preprint arXiv:2202.11649 (2022).

\bibitem{Izgin4} T. Izgin, and P. {\"O}ffner. {\em On the Stability of Modified Patankar Methods}. arXiv preprint arXiv:2206.07371 (2022).

\bibitem{MPRK2_PP} S. Kopecz, A. Meister, \emph{On order conditions for modified Patankar–Runge–Kutta schemes}, Applied Numerical Mathematics, v123 (2018), pp. 159-179.

\bibitem{MPRK3_PP} S. Kopecz, A. Meister, {\em Unconditionally positive and conservative third order modified Patankar–Runge–
Kutta discretizations of production–destruction systems}. BIT Numer. Math. v58 (2018), pp. 691–728.

\bibitem{Lv1} Y. Lv and M. Ihme, {\em Discontinuous Galerkin method for multicomponent chemically reacting flows and combustion}, Journal of Computational Physics, v270 (2014), pp. 105–137.

\bibitem{Lv2} Y. Lv and M. Ihme, {\em High-order discontinuous Galerkin method for applications to multicomponent and chemically reacting flows}, Acta Mechanica Sinica, v33 (2017), pp. 486–499.

\bibitem{DeC} P. {\"O}ffner and D. Torlo, {\em Arbitrary high-order, conservative and positivity preserving Patankar-type deferred correction schemes}. Applied Numerical Mathematics, v153 (2020), pp. 15-34.

\bibitem{Patankar trick} S. Patankar, {\em Numerical Heat Transfer and Fluid Flow}.CRCPress,London(1980)

\bibitem{DG} W. H. Reed and T. R. Hill, {\em Triangular mesh methods for the Neutron transport equation}, Los Alamos Scientific Laboratory Report LA-UR-73-479, Los Alamos, NM, 1973.

\bibitem{Shu-Osher form} C.-W. Shu, S. Osher, {\em Efficient implementation of essentially non-oscillatory shock-capturing schemes}.
Journal of Computational Physics, v77(2), pp. 439–471.

\bibitem{TVD} C.-W. Shu, {\em Total-variation-diminishing time discretizations}, SIAM Journal on Statistical and Scientific Computing, v9
(1988), pp. 1073–1084.


\bibitem{Strang splitting} G. Strang, {\em On the construction and comparison of difference schemes}, SIAM Journal on Numerical Analysis,
v5 (1968), pp. 506–517.


\bibitem{PP/detonation} C. Wang, X. Zhang, C.-W. Shu, and J. Ning, \emph{Robust high order discontinuous Galerkin
schemes for two-dimensional gaseous detonations}, Journal of Computational physics v231 (2012), pp.
653–665


\bibitem{1D example} W. Wang, C.-W. Shu, H.C. Yee and B. Sjogreen, {\em High order well-balanced schemes and
applications to non-equilibrium flow}, Journal of Computational Physics, v228 (2009),
pp. 6682-6702.

\bibitem{rectangular mesh} X. Zhang, C.-W. Shu, \emph{Positivity-preserving high order discontinuous Galerkin schemes for compressible Euler equations with source terms}, Journal of Computational physics, v230 (2011), 1238–1248.

\bibitem{limiter} X. Zhang and C.-W. Shu, \emph{On maximum-principle-satisfying high order schemes for scalar conservation laws}, Journal of Computational Physics, v229 (2010), 3091-3120.

\bibitem{limiter_proof} X. Zhang, \emph{On positivity-preserving high order discontinuous Galerkin schemes for compressible Navier-Stokes equations}, Journal of Computational Physics,  v328 (2017), 301-343.








\end{thebibliography}
\end{document}